\documentclass[12pt]{amsproc}

\usepackage{amssymb,verbatim}
\usepackage{amsfonts}
\usepackage[utf8]{inputenc}
\usepackage[mathscr]{euscript}
\usepackage{enumerate}

\usepackage[
margin=1in,
includefoot,
footskip=30pt,
]{geometry} 

\usepackage{mathtools}
\usepackage{color}

\newtheorem{theorem}[equation]{Theorem}
\newtheorem{lemma}[equation]{Lemma}
\newtheorem{proposition}[equation]{Proposition}
\newtheorem{corollary}[equation]{Corollary}

\theoremstyle{definition}
\newtheorem{definition}[equation]{Definition}

\theoremstyle{remark}
\newtheorem{remark}[equation]{Remark}

\numberwithin{equation}{section}

\newcommand{\newf}[3]{{#1}:{#2}\longrightarrow {#3}}					 

\newcommand{\la}[1]{\text{$\mathcal{#1}$}}
\newcommand{\lb}[1]{\text{$\mathscr{#1}$}}

\newcommand{\powerset}[1]{\text{$\lb{P}(#1)$}}								
\newcommand{\eword}{\text{$\omega$}}											

\newcommand{\xia}{\text{$\xi^\alpha$}}											

\newcommand{\ph}{\text{$\hat{\phi}$}}											

\newcommand{\uset}[1]{\text{$\uparrow\hspace{-0.1cm}{#1}$}}				
\newcommand{\usetr}[2]{\text{$\uparrow_{\scriptscriptstyle{#2}}\hspace{-0.1cm}{#1}$}} 
\newcommand{\lset}[1]{\text{$\downarrow\hspace{-0.1cm}{#1}$}}				
\newcommand{\lsetr}[2]{\text{$\downarrow_{\scriptscriptstyle{#2}}\hspace{-0.1cm}{#1}$}} 

\newcommand{\dgraphg}[1]{\text{$\lb{#1}$}}									
\newcommand{\dgraphupleg}[3]{\text{$\dgraphg{#1}=(\dgraphg{#1}^0,\dgraphg{#1}^1,#2,#3)$}}	

\newcommand{\dgraph}{\text{$\dgraphg{E}$}}									
\newcommand{\dgraphuple}{\text{$\dgraphupleg{E}{r}{s}$}}					
\newcommand{\alfg}[1]{\text{$\lb{#1}$}}										
\newcommand{\acfg}[1]{\text{$\lb{#1}$}}										
\newcommand{\lbfg}[1]{\text{$\lb{#1}$}}										
\newcommand{\acfrg}[1]{\text{$\acfg{B}_{#1}$}}								

\newcommand{\alf}{\text{$\alfg{A}$}}											
\newcommand{\acf}{\text{$\acfg{B}$}}											
\newcommand{\lbf}{\text{$\lbfg{L}$}}											
\newcommand{\acfra}{\text{$\acfg{B}_{\alpha}$}}								
\newcommand{\lgraphg}[2]{\text{$(\dgraphg{#1},\lbfg{#2})$}}				
\newcommand{\lspaceg}[3]{\text{$(\dgraphg{#1},\lbfg{#2},\acfg{#3})$}}		

\newcommand{\lgraph}{\text{$\lgraphg{E}{L}$}}								
\newcommand{\lspace}{\text{$\lspaceg{E}{L}{B}$}}							

\newcommand{\awsetg}[2]{\text{$\lbfg{#1}^{#2}$}}	
\newcommand{\awplus}{\text{$\awsetg{L}{\scriptscriptstyle{\geq 1}}$}}								
\newcommand{\awn}[1]{\text{$\awsetg{L}{#1}$}}								
\newcommand{\awstar}{\text{$\awsetg{L}{\ast}$}}							
\newcommand{\awinf}{\text{$\awsetg{L}{\infty}$}}							
\newcommand{\awleinf}{\text{$\awsetg{L}{\scriptscriptstyle{\leq\infty}}$}}					
\newcommand{\F}{\mathbb{F}}

\newcommand{\fset}[1]{\text{$\textsf{#1}$}}										
\newcommand{\filt}{\text{$\fset{F}$}}											

\newcommand{\ftight}{\text{$\textsf{T}$}}										
\newcommand{\ftightw}[1]{\text{$\textsf{T}_{#1}$}}

\newcommand{\ftg}[1]{\text{$\la{#1}$}}											
\newcommand{\ft}{\text{$\ftg{F}$}}												

\newcommand{\card}[1]{\# #1}

\newcommand{\scj}{\subseteq}

\newcommand{\nn}{\mathbb{N}}

\newcommand{\ssf}{\textsf{X}}

\newcommand{\sat}{\mathcal{S}}

\newcommand{\her}{\mathcal{H}}

\allowdisplaybreaks

\begin{document}

\title[]{Labelled space $C^*$-algebras as partial crossed products and a simplicity characterization }

\author[G. de Castro \and D.W. van Wyk]{Gilles G. de Castro and Daniel W. van Wyk*}
\address{Departamento de Matem\'atica, Universidade Federal de Santa Catarina, 88040-970 Florian\'opolis SC, Brazil.\\ *Current address: Department of Mathematics, Dartmouth College, Hanover, NH 03755 }
\email{gilles.castro@ufsc.br \\ daniel.w.van.wyk@dartmouth.edu}
\keywords{$C^*$-algebra, labelled space, partial actions, partial crossed products, groupoid}
\subjclass[2010]{Primary: 46L55, Secondary: 20M18, 05C20, 05C78}


\begin{abstract}
A partial action is associated with a normal weakly left resolving labelled space such that the crossed product and labelled space $C^*$-algebras are isomorphic. An improved characterization of simplicity for labelled space $C^*$-algebras is given and applied to $C^*$-algebras of subshifts.
\end{abstract}

\maketitle

\section{Introduction}
Examples of $C^*$-algebras have proved very effective in exhibiting abstract structural properties of $C^*$-algebras in a tangible way. Typically, a $C^*$-algebra is associated with an underlying mathematical object such that properties of the $C^*$-algebra is mirrored in the structure of the underlying object. Three such examples that have been extensively studied are directed graphs, dynamical systems and groupoids. A particular property that has been well studied is simplicity. 

In some instances there is more than one model for an associated $C^*$-algebra. 
For example, it is well known that a directed graph has a groupoid model such that their associated $C^*$-algebras are isomorphic \cite{MR1432596,MR1962477}. In \cite{MR3539347} it is shown that a directed graph also has a partial dynamical system associated with it  such that the graph $C^*$-algebra and the crossed product $C^*$-algebra are isomorphic. Simplicity is characterized for graph $C^*$-algebras in \cite{MR1432596,MR1847155,MR1962477,MR2117597},  for partial actions in \cite{MR1905819} and for \'etale groupoids in \cite{MR3189105}. 

Graph $C^*$-algebras generalize Cuntz-Krieger algebras \cite{MR1432596}. Another generalization is Exel-Laca algebras \cite{MR1703078}. In \cite{MR2050134} Tomforde introduces ultragraphs and their $C^*$-algebras as a common framework for Exel-Laca algebras and graph algebras. Simplicity of ultragraph $C^*$-algebras is studied in \cite{MR2001938}. In \cite{MR3938320} a partial action is associated with ultragraphs such that the ultragraph $C^*$-algebra and partial crossed product are isomorphic. A groupoid model for ultragraphs is developed in \cite{MR2457327}. 

Labelled graphs and their $C^*$-algebras are introduced in \cite{MR2304922} as a common framework for  Tomforde's ultragraphs and the shift spaces studied by Matsumoto and  Carlsen \cite{MR2091486,MR2380472}. In \cite{MR3680957} the authors apply Exel's tight spectrum of an inverse semigroup (\cite{MR2419901}) to labelled spaces, and therefore adapt the definition of a labelled space $C^*$-algebra. Independently, the same adapted definition is proposed in \cite{MR3614028}. In this paper we associate a partial action with a labelled space and show that the labelled space $C^*$-algebra and the partial crossed product are isomorphic. Since we define a partial action on the tight spectrum of a labelled space, the appropriate definition of a labelled space $C^*$-algebra is that of \cite{MR3680957} and \cite{MR3614028}. We show that the groupoid associated with a labelled space (as in \cite{Gil3}) is the same groupoid that is obtained from the partial action (as in \cite{MR2045419}). Hence, we show that analogous to graphs and ultragraphs, the $C^*$-algebra of a labelled space has both a groupoid and partial action model. 

We apply our results to characterize simplicity of labelled space $C^*$-algebras in terms of the tight spectrum of the labelled space. Sufficient conditions for simplicity of labelled space $C^*$-algebras are given in \cite{MR2542653}, and a converse is given in \cite{MR2834773,MR3765497}. However, in both instances it is assumed that the labelled space is set-finite and receiver set-finite, the underlying graph has no sinks and no sources, and that the accommodating family of the labelled space is the smallest such family. We do not make any of these assumptions in this paper. In \cite{MR3606190} Boolean dynamical systems are studied, which generalize labelled spaces. A characterization of simplicity for a $C^*$-algebra associated with a Boolean dynamical system is also given in \cite{MR3606190} . In this instance, though, the Boolean dynamical system is assumed to be countable with a certain domain condition. Without assuming countability or any domain conditions we give a new simplicity characterization of labelled space $C^*$-algebras in terms of the tight spectrum of the labelled space and its associated partial action, recovering some of the simplicity results of \cite{MR3606190}.  In particular, we  remove the domain  condition in \cite[Theorem 3.7]{MR3765497}.

In \cite{MR3606190}, the simplicity characterization for the $C^*$-algebra of a Boolean dynamical system is applied to the Cuntz-Pimsner algebra of a one-sided subshift. It is stated that the $C^*$-algebra of the subshift is simple if and only if there is no cyclic point isolated in past equivalence and the subshift is cofinal in past equivalence, \cite[Example 11.4]{MR3606190}.  However, a different simplicity characterization  of these algebras is given in \cite{MR3639522}. We apply our results to $C^*$-algebras of one-sided subshifts and recover the results in \cite{MR3639522}. Then we give an example where the subshift has no cyclic points isolated in past equivalence and is cofinal in past equivalence, but the associated algebra is not simple, which shows that a stronger condition than cofinality is needed in  \cite[Example 11.4]{MR3606190}.

This paper is structured as follows. Section \ref{sec:preliminaries} contains some preliminaries and notation on labelled spaces and their $C^*$-algebras which are used throughout this paper. In Section \ref{s:partialaction} we define a partial action on the tight spectrum of a labelled space. Section \ref{s:crossedprodisom} contains our first main result, Theorem \ref{thm:IsomorphismThm}. Here we show that the partial crossed product, obtained from the partial action of the previous section, is isomorphic to the labelled space $C^*$ -algebra. In section \ref{sec:groupoids} we show that the groupoid associated with a labelled space (as in \cite{Gil3}) is the same groupoid that is obtained from the partial action (as in \cite{MR2045419}). In Section \ref{section:simplicity} we apply our results on the partial action to characterize simplicity of labelled space $C^*$-algebras in terms of the tight spectrum (Theorems \ref{thm:simplicity.tight.spectrum} and \ref{theorem:simple.labelled.space.algebra}). Finally, in Section \ref{sec:subshift}, by applying our  results to subshifts, we  characterize simplicity of certain $C^*$-algebras associated with subshifts, and recover some of the results in \cite{MR3639522}. We also give an example to show that  a stronger condition is needed for simplicity of $C^*$-algebras associated with subshifts than that given in \cite[Example 11.4]{MR3606190}.

\medskip

\noindent\textbf{Acknowledgements:} This study was financed in part by the Coordenação de Aperfeiçoamento de Pessoal de Nível Superior - Brasil (CAPES) - Finance Code 001.

\section{Preliminaries}\label{sec:preliminaries}

\subsection{Filters and characters}\label{subsection:filters.and.characters}

A  \emph{filter} in a partially ordered set $P$ with least element $0$ is a subset $\xi$ of $P$ such that
\begin{enumerate}[(i)]
	\item $0\notin\xi$;
	\item if $x\in\xi$ and $x\leq y$, then $y\in\xi$;
	\item if $x,y\in\xi$, there exists $z\in\xi$ such that $z\leq x$ and $z\leq y$.
\end{enumerate} If $P$ is a (meet) semilattice, condition (iii) may be replaced by $x\wedge y\in\xi$ if $x,y\in\xi$. An \emph{ultrafilter} is a filter which is not properly contained in any filter. 

For  $x\in P$, we define \[\uset{x}=\{y\in P \ | \ x\leq y\}\ \ ,\ \  \lset{x}=\{y\in P \ | \ y\leq x\},\] and for subsets $X,Y$ of $P$ define \[\uset{X} = \bigcup_{x\in X}\uset{x} = \{y\in P \ | \ x\leq y \ \mbox{for some} \ x\in X\},\]
and $\usetr{X}{Y} = Y\cap\uset{X}$; the sets $\usetr{x}{Y}$, $\lsetr{x}{Y}$, $\lset{X}$ and $\lsetr{X}{Y}$ are defined analogously.

\begin{lemma}\label{lemma:ultrafilter}
	Let $P$ be a partially ordered set $P$ with least element $0$. Then, for all $x\in P\setminus\{0\}$, there exists an ultrafilter $\xi$ such that $x\in\xi$.
\end{lemma}

\begin{proof}
	It follows from Zorn's lemma observing that $\uset{x}$ is a filter containing $x$.
\end{proof}

A \emph{lattice} $L$ is a partially ordered set such that every finite subset $\{x_1,\ldots,x_n\}$ has an least upper bound, denoted by $x_1\vee\cdots\vee x_n$, and a greatest lower bound, denoted by $x_1\wedge\cdots\wedge x_n$. If $\ft$ is a filter in a lattice $L$ with least element $0$, we say that $\xi$ is \emph{prime} if for every $x,y\in L$, if $x\vee y\in \xi$, then $x\in\ft$ or $y\in\xi$.

The following result is well known in order theory.

\begin{proposition}
	Let $\xi$ be a filter in a Boolean algebra $\acf$. Then, $\xi$ is an ultrafilter if and only if $\xi$ is a prime filter.
\end{proposition}

\subsection{Labelled spaces}\label{subsect.labelled.spaces}

A \emph{(directed) graph} $\dgraphuple$ consists of non-empty sets $\dgraph^0$ (of \emph{vertices}), $\dgraph^1$ (of \emph{edges}), and \emph{range} and \emph{source} functions $r,s:\dgraph^1\to \dgraph^0$. A vertex $v$ such that $s^{-1}(v)=\emptyset$ is called a \emph{sink}, and the set of all sinks is denoted by $\dgraph^0_{sink}$. The graph is countable if both $\dgraph^0$ and $\dgraph^1$ are countable.

A \emph{path of length $n$} on a graph $\dgraph$ is a sequence $\lambda=\lambda_1\lambda_2\ldots\lambda_n$ of edges such that $r(\lambda_i)=s(\lambda_{i+1})$ for all $i=1,\ldots,n-1$. We write $|\lambda|=n$ for the length of $\lambda$ and regard vertices as paths of length $0$. $\dgraph^n$ stands for the set of all paths of length $n$ and $\dgraph^{\ast}=\cup_{n\geq 0}\dgraph^n$. Similarly, we define a \emph{path of infinite length} (or an \emph{infinite path}) as an infinite sequence $\lambda=\lambda_1\lambda_2\ldots$ of edges such that $r(\lambda_i)=s(\lambda_{i+1})$ for all $i\geq 1$; for such a path, we write $|\lambda|=\infty$ and we let $\dgraph^{\infty}$ denote the set of all infinite paths.

A \emph{labelled graph} consists of a graph $\dgraph$ together with a surjective \emph{labelling map} $\lbf:\dgraph^1\to\alf$, where $\alf$ is a fixed non-empty set, called an \emph{alphabet}, and whose elements are called \emph{letters}. $\alf^{\ast}$ stands for the set of all finite \emph{words} over $\alf$, together with the \emph{empty word} \eword, and $\alf^{\infty}$ is the set of all infinite words over $\alf$. 	We consider $\alf^{\ast}$ as a monoid with operation given by concatenation. In particular, given $\alpha\in \alf^{\ast}\setminus\{\eword\}$ and $n\in\nn^*$, $\alpha^n$ represents $\alpha$ concatenated $n$ times and $\alpha^{\infty}\in \alf^{\infty}$ is $\alpha$ concatenated infinitely many times.

The labelling map $\lbf$ extends in the obvious way to $\lbf:\dgraph^n\to\alf^{\ast}$ and $\lbf:\dgraph^{\infty}\to\alf^{\infty}$. $\awn{n}=\lbf(\dgraph^n)$ is the set of \emph{labelled paths $\alpha$ of length $|\alpha|=n$}, and $\awinf=\lbf(\dgraph^{\infty})$ is the set of \emph{infinite labelled paths}. We consider $\eword$ as a labelled path with $|\eword|=0$, and set $\awplus=\cup_{n\geq 1}\awn{n}$, $\awstar=\{\eword\}\cup\awplus$, and $\awleinf=\awstar\cup\awinf$.  We say that $\alpha$ is a \emph{circuit} if $\alpha^{\infty}=\alpha\alpha\cdots\in\awinf$.

For $\alpha\in\awstar$ and $A\in\powerset{\dgraph^0}$ (the power set of $\dgraph^0$), the \emph{relative range of $\alpha$ with respect to} $A$ is the set
\[r(A,\alpha)=\{r(\lambda)\ |\ \lambda\in\dgraph^{\ast},\ \lbf(\lambda)=\alpha,\ s(\lambda)\in A\}\],
if $\alpha\in\awplus$, and $r(A,\eword)=A$ if $\alpha=\eword$. The \emph{range of $\alpha$}, denoted by $r(\alpha)$, is the set \[r(\alpha)=r(\dgraph^0,\alpha),\]
so that $r(\eword)=\dgraph^0$ and, if $\alpha\in\awplus$, then  $r(\alpha)=\{r(\lambda)\in\dgraph^0\ |\ \lbf(\lambda)=\alpha\}$.

We also define \[\lbf(A\dgraph^1)=\{\lbf(e)\ |\ e\in\dgraph^1\ \mbox{and}\ s(e)\in A\}=\{a\in\alf\ |\ r(A,a)\neq\emptyset\}.\]

A labelled path $\alpha$ is a \emph{beginning} of a labelled path $\beta$ if $\beta=\alpha\beta'$ for some labelled path $\beta'$. Labelled paths  $\alpha$ and $\beta$ are \emph{comparable} if either one is a beginning of the other. If $1\leq i\leq j\leq |\alpha|$, let $\alpha_{i,j}=\alpha_i\alpha_{i+1}\ldots\alpha_{j}$ if $j<\infty$ and $\alpha_{i,j}=\alpha_i\alpha_{i+1}\ldots$ if $j=\infty$. If $j<i$ set $\alpha_{i,j}=\eword$. Define $\overline{\awinf}=\overline{\lbf(\dgraph^{\infty})}=\{\alpha\in\alf^{\infty}\mid \alpha_{1,n}\in\awstar,\forall n\in\nn\}$, that is, it is the set of all infinite words such that all beginnings are finite labelled paths. Clearly $\awinf\scj\overline{\awinf}$, however the inclusion may be proper (see for instance the example below \cite[Definition 2.10]{MR3765497}). Also we write $\overline{\awleinf}=\awstar\cup\overline{\awinf}$.

A \emph{labelled space} is a triple $\lspace$ where $\lgraph$ is a labelled graph and $\acf$ is a family of subsets of $\dgraph^0$ which is closed under finite intersections and finite unions, contains all $r(\alpha)$ for $\alpha\in\awplus$, and is \emph{closed under relative ranges}, that is, $r(A,\alpha)\in\acf$ for all $A\in\acf$ and all $\alpha\in\awstar$. A labelled space $\lspace$ is \emph{weakly left-resolving} if for all $A,B\in\acf$ and all $\alpha\in\awplus$ we have $r(A\cap B,\alpha)=r(A,\alpha)\cap r(B,\alpha)$. A weakly left-resolving labelled space such that $\acf$ is closed under relative complements will be called \emph{normal}\footnote{Note this definition differs from \cite{MR3614028}. However since all labelled spaces considered in this paper are weakly left-resolving, we include `weakly-left resolving' in the definition of a normal labelled space.}. A non-empty set $A\in\acf$ is called \emph{regular} if for all $\emptyset\neq B\scj A$, we have that $0<|\lbf(B\dgraph^1)|<\infty$. The subset of all regular element of $\acf$ together with the empty set is denoted by $\acf_{reg}$.

For $\alpha\in\awstar$, define \[\acfra=\acf\cap\powerset{r(\alpha)}=\{A\in\acf\mid A\scj r(\alpha)\}.\] If a labelled space is normal, then  $\acfra$ is a Boolean algebra for each $\alpha\in\awplus$, and  $\acfrg{\eword}=\acf$ is a generalized Boolean algebra as in \cite{MR1507106}. By the Stone duality, every with $\acfra$, $\alpha\in\awstar$,  is associated with a topological space  $X_{\alpha}$, which consists of the set of ultrafilters in $\acfra$. A basis for $X_{\alpha}$ is given by the family $\{U_A\}_{A\in\acfra}$, where $U_A=\{\ft\in X_{\alpha}|A\in\ft\}$. 

\subsection{The inverse semigroup of a labelled space}\label{subsection:inverse.semigroup}

Let $\lspace$ be normal labelled space and consider the set \[S=\{(\alpha,A,\beta)\ |\ \alpha,\beta\in\awstar\ \mbox{and}\ A\in\acfrg{\alpha}\cap\acfrg{\beta}\ \mbox{with}\ A\neq\emptyset\}\cup\{0\}.\]
Define a  binary operation on $S$  as follows: $s\cdot 0= 0\cdot s=0$ for all $s\in S$ and, if $s=(\alpha,A,\beta)$ and $t=(\gamma,B,\delta)$ are in $S$, then \[s\cdot t=\left\{\begin{array}{ll}
(\alpha\gamma ',r(A,\gamma ')\cap B,\delta), & \mbox{if}\ \  \gamma=\beta\gamma '\ \mbox{and}\ r(A,\gamma ')\cap B\neq\emptyset,\\
(\alpha,A\cap r(B,\beta '),\delta\beta '), & \mbox{if}\ \  \beta=\gamma\beta '\ \mbox{and}\ A\cap r(B,\beta ')\neq\emptyset,\\
0, & \mbox{otherwise}.
\end{array}\right. \]
If  $s=(\alpha,A,\beta)\in S$, we define $s^*=(\beta,A,\alpha)$. Then  $S$ endowed with the operations above is an inverse semigroup with zero element $0$ (\cite{MR3648984}, Proposition 3.4), whose semilattice of idempotents is \[E(S)=\{(\alpha, A, \alpha) \ | \ \alpha\in\awstar \ \mbox{and} \ A\in\acfra\}\cup\{0\}.\]

The natural order in the semilattice $E(S)$ is described in the next proposition.

\begin{proposition}\cite[ Proposition 4.1]{MR3648984}
	Let $p=(\alpha, A, \alpha)$ and $q=(\beta, B, \beta)$ be non-zero elements in $E(S)$. Then $p\leq q$ if and only if $\alpha=\beta\alpha'$ and $A\subseteq r(B,\alpha')$.
\end{proposition}

\subsection{Filters in $\textbf{E(S)}$}\label{subsection:filters.E(S)}

For  a (meet) semilattice $E$ with $0$, there is a bijection between the set of filters in $E$ (upper sets that are closed under meets and that do not contain $0$) and the set $\hat{E}_0$ of characters of $E$ (zero and meet-preserving non-zero maps from $E$ to the Boolean algebra $\{0,1\}$). With the topology of pointwise convergence on $\hat{E}_0$, the closure of the subset $\hat{E}_\infty$ of characters that correspond to ultrafilters in $E$ is denoted by $\hat{E}_{tight}$, and is called the \emph{tight spectrum} of $E$. Elements of $\hat{E}_{tight}$ are the \emph{tight characters of $E$}, and their corresponding filters are \emph{tight filters}. The set of all filters will be denoted by $\filt$ and the set of tight filters will be denoted by $\ftight$, which we also call \emph{tight spectrum}. In particular, we may view $\ftight$ as a closed subspace of $\filt$. See \cite[Section 12]{MR2419901} for details.

Let $\lspace$ be weakly-left resolving labelled space. We recall the description of the tight spectrum of $E(S)$, as given in \cite{MR3648984}\footnote{The results in \cite{MR3648984} are stated in terms of $\awleinf$. This leads to an error in  the proof of \cite[Proposition 4.18]{MR3648984}, where it is assumed that if an infinite word has all beginnings belonging to $\awstar$, then this word is a labelled path. That is, $\awleinf=\overline{\awleinf}$, which is not always the case (see Subsection \ref{subsect.labelled.spaces}). By exchanging $\awleinf$ for $\overline{\awleinf}$ we get the correct descriptions of the filters in \cite{MR3648984}, and the same proofs given in \cite{MR3648984} hold.}. Let $\alpha\in\overline{\awleinf}$  and $\{\ftg{F}_n\}_{0\leq n\leq|\alpha|}$ (understanding that ${0\leq n\leq|\alpha|}$ means $0\leq n<\infty$ when $\alpha\in\overline{\awinf}$) be a family such that $\ftg{F}_n$ is a filter in $\acfrg{\alpha_{1,n}}$ for every $n>0$ and $\ftg{F}_0$ is either a filter in $\acf$ or $\ftg{F}_0=\emptyset$. The family $\{\ftg{F}_n\}_{0\leq n\leq|\alpha|}$ is  a  \emph{complete family for} $\alpha$ if
\[\ftg{F}_n = \{A\in \acfrg{\alpha_{1,n}} \ | \ r(A,\alpha_{n+1})\in\ftg{F}_{n+1}\}\]
for all $n\geq0$.

\begin{theorem}\cite[Theorem 4.13]{MR3648984}\label{thm.filters.in.E(S)}
	Let $\lspace$ be a weakly left-resolving labelled space and $S$   its associated inverse semigroup. Then there is a bijective correspondence between filters in $E(S)$ and pairs $(\alpha, \{\ftg{F}_n\}_{0\leq n\leq|\alpha|})$, where $\alpha\in\overline{\awleinf}$ and $\{\ftg{F}_n\}_{0\leq n\leq|\alpha|}$ is a complete family for $\alpha$.
\end{theorem}

Filters are of \emph{finite type} if they are associated with pairs $(\alpha, \{\ftg{F}_n\}_{0\leq n\leq|\alpha|})$ for which $|\alpha|<\infty$, and of \emph{infinite type} otherwise.

A filter $\xi$ in $E(S)$ with associated labelled path $\alpha\in\overline{\awleinf}$ is sometimes denoted by $\xia$ to stress the word $\alpha$; in addition, the filters in the complete family associated with $\xia$ will be denoted by $\xi^{\alpha}_n$ (or simply $\xi_n$). Specifically,
\begin{align}\label{eq.defines.xi_n}
\xi^{\alpha}_n=\{A\in\acf \ | \ (\alpha_{1,n},A,\alpha_{1,n}) \in \xia\}.
\end{align}

\begin{remark}\label{remark.when.in.xialpha}
	It follows from \cite[Propositions 4.4 and 4.8]{MR3648984} that for a filter $\xi^\alpha$ in $E(S)$ and an element $(\beta,A,\beta)\in E(S)$ we have that $(\beta,A,\beta)\in\xi^{\alpha}$ if and only if $\beta$ is a beginning of $\alpha$ and $A\in\xi_{|\beta|}$.
\end{remark}

\begin{theorem}[\cite{MR3648984}, Theorems 5.10 and 6.7] \label{thm:TightFiltersType}
	\label{thm.tight.filters.in.es}
	Let $\lspace$ be a normal labelled space and $S$  its associated inverse semigroup. Then the tight filters in $E(S)$ are:
	\begin{enumerate}[(i)]
		\item The filters of infinite type for which the non-empty elements of their associated complete families are ultrafilters. 
		\item The filters of finite type $\xia$ such that $\xi_{|\alpha|}$ is an ultrafilter in $\acfra$ and for each  $A\in\xi_{|\alpha|}$ at least one of the following conditions hold:
		\begin{enumerate}[(a)]
			\item $\lbf(A\dgraph^1)$ is infinite.
			\item There exists $B\in\acfra$ such that $\emptyset\neq B\subseteq A\cap \dgraph^0_{sink}$.
		\end{enumerate}
	\end{enumerate}
\end{theorem}

For $\alpha\in\awstar$, we denote by $\ftightw{\alpha}$ the set of all tight filters in $E(S)$ for which the associated word is $\alpha$.

\subsection{Filter surgery in E(S)}\label{subsection.filter.surgery}

Fix a labelled space  $\lspace$. Let $X_\alpha$ be the topological space associated with the Boolean algebra $\acfra$ via Stone's duality. We give an outline of glueing and cutting of filters, and refer the reader to \cite[Section 4]{MR3680957} for more details and the proofs.

If $\alpha,\beta\in\awplus$ such that $\alpha\beta\in\awplus$, the relative range map $\newf{r(\,\cdot\,,\beta)}{\acfrg{\alpha}}{\acfrg{\alpha\beta}}$ is a morphism of Boolean algebras and, therefore, we have its dual morphism \[\newf{f_{\alpha[\beta]}}{X_{\alpha\beta}}{X_{\alpha}}\] given by
\begin{equation}\label{eq:def.f}
f_{\alpha[\beta]}(\ft)=\{A\in\acfra \ | \ r(A,\beta)\in\ft\}.
\end{equation}
 If $\alpha=\eword$ and $\ft\in\acfrg{\beta}$, then $f_{\eword[\beta]}(\ft)=\{A\in\acf \ | \ r(A,\beta)\in\ft\}$ is either an ultrafilter in $\acf=\acfrg{\eword}$ or is the empty set, and we can therefore consider $\newf{f_{\eword[\beta]}}{X_{\beta}}{X_{\eword}\cup\{\emptyset\}}$. The functions $f_{\alpha[\beta]}$ are continuous  and $f_{\alpha[\beta\gamma]}=f_{\alpha[\beta]}\circ f_{\alpha\beta[\gamma]}$.

We now review functions described in \cite{MR3680957} that are used throughout this paper.

We begin with the  ``gluing map'': for composable labelled paths $\alpha\in\awplus$ and $\beta\in\awstar$ (that is, such that $\alpha\beta\in\awplus$), consider the subspace $X_{(\alpha)\beta}$ of $X_\beta$ given by \[X_{(\alpha)\beta}=\{\ft\in X_\beta \ |\ r(\alpha\beta)\in\ft\}.\] Then there is a continuous map \[\newf{g_{(\alpha)\beta}}{X_{(\alpha)\beta}}{X_{\alpha\beta}}\] on ultrafilters induced by gluing $\alpha$ at the beginning of the labelled path $\beta$ given by
\begin{align}\label{eq:def.g}
g_{(\alpha)\beta}(\ft)=\{C\cap r(\alpha\beta)\ |\ C\in\ft\}.
\end{align}

For labelled paths $\alpha\in\awplus$ and $\beta\in\overline{\awleinf}$, let $\ftightw{(\alpha)\beta}$ be the subspace of $\ftightw{\beta}$ given by \[ \ftightw{(\alpha)\beta}=\{\xi\in\ftightw{\beta}\ |\ \xi_0\in X_{(\alpha)\eword}\}.\]
Then we define a gluing map \[\newf{G_{(\alpha)\beta}}{\ftightw{(\alpha)\beta}}{\ftightw{\alpha\beta}},\] taking a tight filter $\xi\in\ftightw{(\alpha)\beta}$ to the tight filter $\eta\in\ftightw{\alpha\beta}$, whose complete family of (ultra) filters is obtained by gluing and cutting labelled paths  as follows:
\begin{itemize}
	\item If $\beta=\eword$, \[ \eta_{|\alpha|}=g_{(\alpha)\eword}(\xi_0)=\{C\cap r(\alpha)\ |\ C\in\xi_0\} \] and,  for $0\leq i< |\alpha|$, \[ \eta_i=f_{\alpha_{1,i}[\alpha_{i+1,|\alpha|}]}(\eta_{|\alpha|})=\{D\in\acfrg{\alpha_{1,i}}\ |\ r(D,\alpha_{i+1,|\alpha|})\in\eta_{|\alpha|}\};\]
	
	\item If $\beta\neq\eword$, for $1\leq n\leq|\beta|$ (or $n<|\beta|$ if $\beta$ is infinite)
	\[ \eta_{|\alpha|+n} = g_{(\alpha)\beta_{1,n}}(\xi_n) = \{C\cap r(\alpha\beta_{1,n})\ |\ C\in\xi_n\}\] and, for $0\leq i\leq|\alpha|$,
	\[\eta_i = f_{\alpha_{1,i}[\alpha_{i+1,|\alpha|}\beta_1]}(\eta_{|\alpha|+1})
	=\{D\in\acfrg{\alpha_{1,i}}\ |\ r(D,\alpha_{i+1,|\alpha|}\beta_1)\in\eta_{|\alpha|+1}\}.\]
\end{itemize}
Finally, for $\alpha=\eword$ set $\ftightw{(\eword)\beta}=\ftightw{\beta}$ and let $G_{(\eword)\beta}$ be the identity function on $\ftightw{\beta}$.

\begin{remark}
	In \cite{MR3680957}, it is asked that $\alpha$ and $\beta$ are composable in order for $G_{(\alpha)\beta}$ to make sense when $\beta\neq\eword$. However, if $\xi\in\ftight_{(\alpha)\beta}$, then $\alpha$ and $\beta$ are indeed composable. To see this, suppose that $\xi^{\beta}\in \ftightw{(\alpha)\beta}$. Then, by the definition of complete family, for $1\leq n\leq|\beta|$ (or $n<|\beta|$ if $\beta$ is infinite), $r(\alpha\beta_{1,n})=r(r(\alpha),\beta_{1,n})\in\xi_n$. Since $\xi_n$ is a filter, then $r(\alpha\beta_{1,n})\neq\emptyset$, and hence $\alpha\beta_{1,n}$ must be a labelled path. Since $n$ was arbitrary, we have that $\alpha\beta\in\overline{\awleinf}$.
\end{remark} 

Next, we describe the ``cutting map": for composable labelled paths $\alpha\in\awplus$ and $\beta\in\awstar$, there is a continuous map \[\newf{h_{[\alpha]\beta}}{X_{\alpha\beta}}{X_{(\alpha)\beta}}\] induced on ultrafilters by cutting $\alpha$ from the beginning of $\alpha\beta$ given by
\begin{align}\label{eq:def.h}
h_{[\alpha]\beta}(\ft)=\usetr{\ft}{\acfrg{\beta}}=\{C\in\acfrg{\beta}\ |\ D\scj C\text{ for some }D\in\ft\}. 
\end{align}

For composable labelled paths $\alpha\in\awplus$ and $\beta\in\overline{\awleinf}$, this map gives rise to a cutting map \[\newf{H_{[\alpha]\beta}}{\ftightw{\alpha\beta}}{\ftightw{(\alpha)\beta}}\] that takes a tight filter $\xi\in\ftightw{\alpha\beta}$ to the tight filter $\eta\in\ftightw{(\alpha)\beta}$ such that, for all $n$ with $0\leq n\leq |\beta|$, \[ \eta_n=h_{[\alpha]\beta_{1,n}}(\xi_{n+|\alpha|}).\] For $\alpha=\eword$ define $H_{[\eword]\beta}$ to be the identity function on $\ftightw{\beta}$.


\begin{theorem}\cite[Theorem 4.17]{MR3680957}
	\label{thm.H-class.G-class.inverses}
	Let  $\lspace$ be a normal labelled space, and let $\alpha\in\awplus$ and $\beta\in\overline{\awleinf}$ be such that $\alpha\beta\in\overline{\awleinf}$. Then $H_{[\alpha]\beta}=(G_{(\alpha)\beta})^{-1}$.
\end{theorem}

\begin{theorem}\cite[Lemmas 4.13 and 4.16]{MR3680957}
	\label{thm.compositions.G.and.H}
	Let  $\lspace$ be a normal labelled space, and let $\alpha,\beta\in\awplus$ and $\gamma\in\overline{\awleinf}$ be such that $\alpha\beta\gamma\in\overline{\awleinf}$. Then $G_{(\alpha\beta)\gamma}=G_{(\alpha)\beta\gamma}\circ G_{(\beta)\gamma}$ and $H_{[\beta]\gamma}\circ H_{[\alpha]\beta\gamma}=H_{[\alpha\beta]\gamma}$.
\end{theorem}

\subsection{The $C^*$-algebra of a labelled space}\label{sec:labelCstarAlgDfn}
Let $\lspace$ be a normal labelled space. The \emph{C*-algebra associated with} $\lspace$, denoted by $C^*\lspace$, is the universal $C^*$-algebra generated by projections $\{p_A \ | \ A\in \acf\}$ and partial isometries $\{s_a \ | \ a\in\alf\}$ subject to the relations
\begin{enumerate}[(i)]
	\item $p_{A\cap B}=p_Ap_B$, $p_{A\cup B}=p_A+p_B-p_{A\cap B}$ and $p_{\emptyset}=0$, for every $A,B\in\acf$;
	\item $p_As_a=s_ap_{r(A,a)}$, for every $A\in\acf$ and $a\in\alf$;
	\item $s_a^*s_a=p_{r(a)}$ and $s_b^*s_a=0$ if $b\neq a$, for every $a,b\in\alf$;
	\item For every $A\in\acf$ for which $0<\card{\lbf(A\dgraph^1)}<\infty$ and there does not exist $B\in\acf$ such that $\emptyset\neq B\subseteq A\cap \dgraph^0_{sink}$,,
	\[p_A=\sum_{a\in\lbf(A\dgraph^1)}s_ap_{r(A,a)}s_a^*.\]
\end{enumerate}

For each word $\alpha=a_1a_2\cdots a_n$, define $s_\alpha=s_{a_1}s_{a_2}\cdots s_{a_n}$ and set $s_{\eword}=1$ for the the empty word $\eword$ . Observe that $s_\eword$ does not belong to $C^*\lspace$ unless it is unital -- we work with $s_{\eword}$ to simplify our statements. For example, $s_{\eword}p_{A}s_{\eword}^*$ means $p_{A}$. We never use $s_{\eword}$ alone.

\begin{proposition}\label{prop:closure.span}
	Let $\lspace$ be a  normal labelled space. Then \[C^*\lspace = \overline{\mathrm{span}}\{s_\alpha p_A s_\beta^* \ | \ \alpha,\beta\in\awstar \ \mbox{and} \ A\in\acfrg{\alpha}\cap\acfrg{\beta}\}.\]
\end{proposition}

For details, see \cite{MR3614028} and \cite{MR3680957}.

\section{A partial action on the tight filters of $E(S)$} \label{s:partialaction}
In this section we define a partial action of the free group generated by $\alf$ on the tight spectrum $\ftight$ of a labelled space. Our construction is in the same spirit as that of graphs \cite{MR3539347}. However, due to the one-to-one correspondence between filters in $\ftight$ and pairs consisting of a labelled path and a family of filters (see Theorem \ref{thm.filters.in.E(S)}), there is an extra layer of complexity that needs to dealt with.

We recall the definition of a topological semi-saturated orthogonal partial action:
\begin{definition}\cite[Section 2]{MR1703078}
	A \textit{partial action} of a group $G$ on a topological space $X$ is pair  $$ \Phi=(\{V_t\}_{t\in G},\{\phi_t\}_{t\in G}) $$ 
	consisting of open sets $\{V_t\}_{t\in G}$ and homeomorphisms $\phi_t:V_{t^{-1}}\to V_t$ such that 
	\begin{enumerate}
		\item $V_{e}=V_{e^{-1}}=X$ and $\phi_e$ is the identity on $X$,
		\item $\phi_s(V_{s^{-1}}\cap V_{t})= V_s\cap V_{st}$, and 
		\item $\phi_s(\phi_t(x))=\phi_{st}(x)$ for every $x \in V_{t^{-1}}\cap V_{(st)^{-1}}$.
	\end{enumerate}
	If the partial action is given by the free group $\F$ on a set of generators, then the partial action is \textit{semi-saturated} if 
	$$ \phi_s\circ\phi_t=\phi_{st} $$  
	for every $s,t\in\F$ such that $|st|=|s|+|t|$, and \textit{orthogonal} if $V_s\cap V_t =\emptyset$ for  $s\neq t$.
\end{definition}

Fix a weakly-left resolving labelled space $\lspace$.
We begin by describing the topology on $\ftight$.  For $e\in E(S)$ define
\[ U_e=\{\xi\in\filt\ |\ e\in\xi\}. \]
If  $\{e_1,\ldots,e_n\}$ is a finite (possibly empty) set in $E(S)$, define
\[ U_{e:e_1,\ldots,e_n}= U_e\cap U_{e_1}^c\cap\cdots\cap U_{e_n}^c.\]
%
\begin{proposition}\cite[Lemma 2.22 and Lemma 2.23]{MR2974110}
	The sets $U_{e:e_1,\ldots,e_n}$ form a basis of compact-open sets for a Hausdorff topology  on $\filt$.
\end{proposition}

\begin{corollary}\label{cor:basis.tight}
	For $e\in E(S)$ and a finite (possibly empty) set $\{e_1,\ldots,e_n\}\ \subset E(S)$, the sets
	\[  V_{e:e_1,\ldots,e_n} :=  U_{e:e_1,\ldots,e_n}\cap\ftight\] 
	form a basis of compact-open sets for a Hausdorff topology  on $\ftight$.
\end{corollary}
\begin{proof}
	The result follows from the fact that $\ftight$ is a closed subset of $\filt$. 
\end{proof}
In particular, we denote $V_e=U_e\cap\ftight$.

We next describe the open domains and codomains for the partial action on $\ftight$. Let $\alpha\in \awstar$. Then 
\begin{equation*}
\bigcup \{\ftight_{(\alpha)\beta}\mid \beta\in\overline{\awleinf} \text{ and } \alpha\beta\in\overline{\awleinf}\} 
\end{equation*}
is the set of all filters in $E(S)$ whose associated labelled path can be glued  to $\alpha$, and 
\begin{equation*}
\bigcup \{\ftight_{\alpha\beta}\mid \beta\in\overline{\awleinf} \text{ and } \alpha\beta\in\overline{\awleinf}\}
\end{equation*}
is the set of all filters in $E(S)$ whose associated word begins with $\alpha$. Note that if $\beta,\gamma\in\overline{\awleinf}$ and $\beta\neq\gamma$, then 	$\ftight_{\alpha\beta}\cap\ftight_{\alpha\gamma}=\emptyset$ and $\ftight_{(\alpha)\beta}\cap\ftight_{(\alpha)\gamma}=\emptyset$ for all $\alpha\in\awstar$. Hence the unions above are disjoint unions.
To simplify notation we write $\bigsqcup_{\beta}\ftight_{(\alpha)\beta}$ where it is understood that the union is taken over all $\beta\in\overline{\awleinf}$ such that $\alpha\beta\in\overline{\awleinf}$. 
Note that $\bigsqcup_{\beta} \ftight_{\eword\beta}=\bigsqcup_{\beta} \ftight_{(\eword)\beta} =\ftight$.

\begin{lemma}\label{lem:compactopenneigh}
	Fix $\alpha \in\awplus$. Then 
	\begin{enumerate}[(i)]
		\item $V_{(\alpha,r(\alpha),\alpha)}=\bigsqcup_{\beta}\ftight_{\alpha\beta}$, and 
		\item  $V_{(\eword,r(\alpha),\eword)}=\bigsqcup_{\beta}\ftight_{(\alpha)\beta}$.
	\end{enumerate}
\end{lemma}
\begin{proof}
	(i) Assume $\xi\in \bigsqcup_{\beta}\ftight_{\alpha\beta}$ and $|\alpha|=n$. Then there is a $\beta\in\overline{\awleinf}$ such that $\alpha\beta\in\overline{\awleinf}$ and $\xi=\xi^{\alpha\beta}$ with associated with a pair $(\alpha\beta, \{\xi^{\alpha\beta}_i\}_{i\geq0})$. Since $r(\alpha_{1,n})$ is contained in any filter $\mathcal{F}_n\subset\acfrg{\alpha_{1,n}}$, it follows from  Remark \ref{remark.when.in.xialpha}  that 
	$$(\alpha,r(\alpha),\alpha)=(\alpha_{1,n},r(\alpha_{1,n}),\alpha_{1,n})\in\xi^{\alpha\beta}.$$
	Hence $\xi=\xi^{\alpha\beta}\in V_{(\alpha,r(\alpha),\alpha)}$. 
	
	Now suppose that $\xi\in V_{(\alpha,r(\alpha),\alpha)}$ and $\xi=\xi^{\gamma}$ for some $\gamma\in \overline{\awleinf}$. Then 
	$$(\alpha,r(\alpha),\alpha)=(\alpha_{1,n},r(\alpha_{1,n}),\alpha_{1,n})\in\xi^{\gamma},$$
	and it follows from Remark \ref{remark.when.in.xialpha} that $\gamma =\alpha\beta$ for some $\beta\in\overline{\awleinf}$. Thus $\xi\in\bigsqcup_{\beta}\ftight_{\alpha\beta}$.
	
	(ii) Suppose $\xi^{\beta}\in V_{(\eword,r(\alpha),\eword)}$ and let $\{\xi_n^{\beta}\}_{n\geq 0}$ be the complete family of filters for $\beta$ associated with $\xi^{\beta}$. Then 
	$$ (\eword,r(\alpha),\eword)=(\beta_{1,0},r(\alpha),\beta_{1,0})\in \xi^{\beta},$$
	which implies that 
	$$r(\alpha)\in\xi_0^{\beta}=\{A\in\acf_{\beta_{1,0}}\mid r(A,\beta_{1,m})\in\xi^{\beta}_{m}\},$$
	for all $m>0$. 
	Therefore $r(r(\alpha),\beta_m)\in \xi^{\beta}_{m}$ for all $m>0$. Since $\xi_m$ is a filter it is non-empty. Thus $r(r(\alpha),\beta_m)\neq \emptyset$, which shows that $\alpha\beta\in\overline{\awleinf}$. 	
	Hence,  if $r(\alpha)\in\xi_0^{\beta}$ then  $\xi_0^{\beta}\in X_{(\alpha)\eword}$, and thus $\xi^{\beta}\in \ftight_{(\alpha)\beta}$. Hence $V_{(\eword,r(\alpha),\eword)}\subseteq\bigsqcup_{\beta}\ftight_{(\alpha)\beta}$. 
	
	To see that $\bigsqcup_{\beta}\ftight_{(\alpha)\beta}\subseteq V_{(\eword,r(\alpha),\eword)}$, suppose that $\xi^{\beta}\in \ftight_{(\alpha)\beta}$ for some $\beta\in\overline{\awleinf}$. Then 
	$\xi_0^{\beta}\in X_{(\alpha)\eword}$, which implies that $r(\alpha)\in\xi_0^{\beta}$. Therefore
	$(\beta_{1,0},r(\alpha),\beta_{1,0})\in \xi^{\beta}$, and thus   $\xi^{\beta}\in V_{(\eword,r(\alpha),\eword)}$, completing the proof.
\end{proof}

Next we define a partial action of the free group $\mathbb{F}$ generated by $\alfg{A}$ (identifying the identity of $\mathbb{F}$ with $\eword$) on $\ftight$. 
Let $a\in \alfg{A}$. Put
\begin{equation*}
\begin{split}
&V_\eword =\ftight, \\
&V_a:=V_{(a,r(a),a)}, \text{ and }  \\
&V_{a^{-1}}:= V_{(w,r(a),w)}.
\end{split}
\end{equation*}
Define 
$\phi_a:V_{a^{-1}}\to V_a$ by   
\begin{equation} \label{partialactionDFN1}
\phi_{a} |_{\ftight_{(a)\beta}}= G_{(a)\beta},
\end{equation}
and $\phi_{a}^{-1}:V_{a}\to V_{a^{-1}}$ by 
\begin{equation} \label{partialactionDFN2}
\phi_{a}^{-1}|_{\ftight_{a\beta}}= H_{[a]\beta},
\end{equation}
where $G$ and $H$ are the gluing and cutting maps defined in Section \ref{subsection.filter.surgery}.
For the empty word we define $\phi_w$ to be the identity map  $\mathrm{id}_{\ftight}$ on $\ftight$. 
Let $t\in\mathbb{F}$ and suppose that $t=a_n\cdots a_1$ is the reduced form of $t$, with each $a_i\in \alfg{A}\cup\alfg{A}^{-1}$ and $n\geq 2$. We extend the definitions above to  $\phi_t$ inductively as follows: let 
\begin{equation}\label{e:toppartialactionextend}
\begin{split}
& V_{t^{-1}}=V_{(a_n\cdots a_1)^{-1}}=\phi^{-1}_{a_{n-1}\cdots a_1}(V_{a_n^{-1}}),\text{ and } \\
&  \phi_{t}(\xi)=\phi_{a_n\cdots a_1}(\xi)=\phi_{a_n}(\phi_{a_{n-1}\cdots a_1}(\xi)), 
\end{split}
\end{equation} 
for  $\xi\in V_{(a_n\cdots a_1)^{-1}}$.
\begin{remark}\label{rem:partialmapshomeomorphisms}
	Since $G_{(a)\beta}$ and $H_{[a]\beta}$ are inverses of each other (Theorem \ref{thm.H-class.G-class.inverses}), if  $a\in\alfg{A}$ then $\phi_a$ and $\phi_{a^{-1}}$ are bijections and $\phi_a^{-1}=\phi_{a^{-1}}$. Suppose $\alpha\beta\in \F$ is in reduced form and $\alpha=a_1\cdots a_m$ and $\beta=b_1\cdots b_n$ with each $a_i,b_j\in \alfg{A}\cup\alfg{A}^{-1}$.  Then by Equation \ref{e:toppartialactionextend} we have that
	$$\phi_{\alpha\beta}=\phi_{a_1}\circ\cdots \circ\phi_{a_m}\circ\phi_{b_1}\circ\cdots \circ\phi_{b_n}= \phi_{\alpha}\circ\phi_{\beta}$$ 
	on the appropriate domain $V_{(\alpha\beta)^{-1}}$. Moreover, note that $V_{(\alpha\beta)^{-1}}=\phi_{\beta}^{-1}(V_{\alpha^{-1}})\scj V_{\beta^{-1}}.$
\end{remark}
In Proposition \ref{prop:partialtopaction}, we show that the maps defined in (\ref{partialactionDFN1}), (\ref{partialactionDFN2}) and (\ref{e:toppartialactionextend}) give a partial action of $\mathbb{F}$ on $\ftight$. For this we need the following lemmas.

\begin{lemma}\label{lem:domainpart1}
	 If $\alpha\in \awplus$, then
	$V_{\alpha^{-1}}=V_{(\eword,r(\alpha),\eword)}$,  $V_\alpha=V_{(\alpha,r(\alpha),\alpha)}$ and 
	$ \phi_\alpha(\xi^{\beta})= G_{(\alpha)\beta}(\xi^{\beta})$ for every $\xi^{\beta}\in V_{\alpha^{-1}}$.
\end{lemma}

\begin{proof}
	If $\alpha\in \awplus$, then $\alpha=a_n\cdots a_1$ with $a_i\in \alf$ and $1\leq n< \infty$. By Lemma \ref{lem:compactopenneigh}, to show that   $V_{\alpha^{-1}}=V_{(\eword,r(\alpha),\eword)}$ it would suffice to show that $V_{\alpha^{-1}}=\bigsqcup_{\beta}\ftight_{(\alpha)\beta}$ for all  $1\leq n <\infty$. We show this by induction. If $n=1$, then $\alpha=a_1$ is a letter in $\alf$, and $V_{a_1^{-1}}=V_{(\eword,r(a_1),\eword)}$ by definition. Suppose the statement is true for $n=k$, that is, 
	$$ V_{(a_k\cdots a_1)^{-1}} =\bigsqcup_{\gamma}\ftight_{(a_k\cdots a_1)\gamma}. $$
	We show that the statement holds for $n=k+1$: 
	\begin{eqnarray*}
		V_{(a_{k+1}\cdots a_1)^{-1}} &=& \phi_{a_{k}\cdots a_1}^{-1}(V_{a_{k+1}^{-1}})   \\
		&=& \phi_{a_1^{-1}}\left( \phi_{a_{k}\cdots a_2}^{-1}(V_{a_{k+1}^{-1}})\right) \\
		&=&  \phi_{a_1^{-1}}\left(V_{(a_{k+1}\cdots a_2)^{-1}} \right) \\
		&=&  \phi_{a_1^{-1}}\left(  \bigsqcup_{\gamma}\ftight_{(a_{k+1}\cdots a_2)\gamma}\right) \\
		&=& \phi_{a_1^{-1}}\left(  \bigsqcup_{\gamma} H_{[a_{k+1}\cdots a_2]\gamma} \left(\ftight_{a_{k+1}\cdots a_2\gamma}\right)\right)  \\	
		&=&  \bigsqcup_{\beta} H_{[a_1]\beta}\circ H_{[a_{k+1}\cdots a_2]a_1\beta} \left(\ftight_{a_{k+1}\cdots a_2a_1\beta}\right)  \\
		&=&  \bigsqcup_{\beta} H_{[a_{k+1}\cdots a_1]\beta} \left(\ftight_{a_{k+1}\cdots a_1\beta}\right)  \\
		&=&  \bigsqcup_{\beta} \ftight_{(a_{k+1}\cdots a_1)\beta},
	\end{eqnarray*} 
	where the fourth and sixth equalities follow from the induction hypothesis and Theorem \ref{thm.compositions.G.and.H}, respectively. 
	Hence $V_{\alpha^{-1}}=V_{(\eword,r(\alpha),\eword)}$ for all  $1\leq n <\infty$.  A similar induction argument as above together with Theorem \ref{thm.compositions.G.and.H} shows that $V_{\alpha}=\bigsqcup_{\beta}\ftight_{\alpha\beta}$, and thus by Lemma 
	\ref{lem:compactopenneigh}, we have that  $V_\alpha=V_{(\alpha,r(\alpha),\alpha)}$. 
	
	Let  $\xi^{\beta}\in V_{\alpha^{-1}}=\bigsqcup_{\beta}\ftight_{(\alpha)\beta}$. Then, by Theorem \ref{thm.compositions.G.and.H} and Remark \ref{rem:partialmapshomeomorphisms}, $ \phi_\alpha(\xi^{\beta})=G_{(\alpha)\beta}(\xi^{\beta})$, completing the proof.
\end{proof}

\begin{lemma}\label{lem:intersectionproperties}
	Let $\alpha,\beta\in\awplus$ and suppose that $\alpha\beta^{-1}\in \mathbb{F}$ is in reduced form. Then the following are equivalent:
	\begin{enumerate}
		\item[(i)] $r(\alpha)\cap r(\beta)\neq\emptyset$,
		\item[(ii)] $V_{\alpha^{-1}}\cap V_{\beta^{-1}}  \neq\emptyset$, \text{ and}
		\item[(iii)]  $V_{(\alpha\beta^{-1})^{-1}}\neq\emptyset$.
	\end{enumerate}
\end{lemma}
\begin{proof}
	We first show the equivalence of (i) and (ii), and then that (ii) is equivalent to (iii) by using  the description of $V_{\alpha^{-1}}$ given in Lemma \ref{lem:domainpart1}.
	
	(i)$\Rightarrow$(ii): Assume $r(\alpha)\cap r(\beta)\neq\emptyset$. Then $(\eword,r(\alpha)\cap r(\beta),\eword)\in E(S)$. By Zorn's lemma there is an ultra filter $\xi\in \ftight$ such that $(\eword,r(\alpha)\cap r(\beta),\eword)\in\xi$. If $\ft\subset \acfrg{\alpha}\cap\acfrg{\beta}$ is any filter such that $r(\alpha)\cap r(\beta)\in\ft$, then $r(\alpha)\in\ft$ and $r(\beta)\in\ft$, since $r(\alpha)\cap r(\beta)\subset r(\alpha)$ and $r(\alpha)\cap r(\beta)\subset r(\beta)$.
	Hence $(\eword,r(\alpha),\eword)\in\xi$ and $(\eword, r(\beta),\eword)\in\xi$, implying that $\xi\in V_{\alpha^{-1}}\cap V_{\beta^{-1}}$.

	(ii)$\Rightarrow$(i): Conversely assume that $V_{\alpha^{-1}}\cap V_{\beta^{-1}}  \neq\emptyset$.
	Then there is a filter $\xi\in \ftight$ such that 
	$(\eword,r(\alpha),\eword)\in\xi$ and $(\eword,r(\beta),\eword)\in\xi$. Since $E(S)$ is a semilattice, it follows that  
	$(\eword,r(\alpha)\cap r(\beta),\eword)\in\xi$, and that $r(\alpha)\cap r(\beta)\neq\emptyset$.
	
	(ii)$\Rightarrow$(iii): By Remark \ref{rem:partialmapshomeomorphisms}, $V_{(\alpha\beta^{-1})^{-1}}=\phi^{-1}_{\beta^{-1}}(V_{\alpha^{-1}})$. If $\xi\in V_{\alpha^{-1}}\cap V_{\beta^{-1}}$, then $\phi_{\beta}(\xi)$ is such that $\phi_{\beta^{-1}}(\phi_{\beta}(\xi))=\xi\in V_{\alpha^{-1}}$, that is, $\phi_{\beta}(\xi)\in V_{(\alpha\beta^{-1})^{-1}}$.
	
	(iii)$\Rightarrow$(ii): For $\xi\in V_{(\alpha\beta^{-1})^{-1}}$, we have that $\phi_{\beta^{-1}}\in V_{\alpha^{-1}}\cap V_{\beta^{-1}}$.

\end{proof}

\begin{lemma}\label{lem:domainsofreducedwords}
	Let  $t\in\F$ be in reduced form. Then $V_t$ and $\phi_t$ as defined in (\ref{e:toppartialactionextend}) satisfy the following:
	\begin{enumerate}
		\item[(i)] If $\alpha=a_1\cdots a_n, \beta=b_1\cdots b_m\in\awplus$ such that 
		$\alpha,\beta\neq \eword$, $a_n\neq b_m$ and $r(\alpha)\cap r(\beta)\neq \emptyset$, then 
		$\emptyset\neq V_{\beta\alpha^{-1}}\subseteq V_{\beta}$, 
		and $\phi_{\alpha\beta^{-1}}(\xi^{\beta\gamma})=G_{(\alpha)\gamma}\circ H_{[\beta]\gamma}(\xi^{\beta\gamma})$ for every
		$ \xi^{\beta\gamma}\in V_{\beta\alpha^{-1}}$.
		\item[(ii)] If $t\notin \{\eword\}\cup\{\alpha\mid\alpha\in\awplus\}\cup\{\alpha^{-1}\mid\alpha\in\awplus\}\cup\{\beta\alpha^{-1}\mid \beta,\alpha\in\awplus, r(\alpha)\cap r(\beta)\neq\emptyset\}$, then $V_t=V_{t^{-1}}=\emptyset$.
	\end{enumerate}
\end{lemma}

\begin{proof}
	(i) If $r(\alpha)\cap r(\beta)\neq \emptyset$, then $\emptyset\neq V_{\beta\alpha^{-1}}\subseteq V_{\beta}$ by Lemma \ref{lem:intersectionproperties}. We show that $\phi_{\alpha\beta^{-1}}(\xi^{\beta\gamma})=G_{(\alpha)\gamma}\circ H_{[\beta]\gamma}(\xi^{\beta\gamma})$.
	Note that $\phi_{\alpha\beta^{-1}}: V_{\beta\alpha^{-1}}\to V_{(\beta\alpha^{-1})^{-1}}$. Since  $\emptyset\neq V_{\beta\alpha^{-1}}$, there exists a
	$\xi\in V_{\beta\alpha^{-1}}$. Moreover, since $V_{(\alpha\beta^{-1})^{-1}}=\phi_{\beta}(V_{{\alpha}^{-1}})$ it follows that $\xi=\xi^{\beta\gamma}$ with $\gamma$ such that $H_{[\beta]\gamma}(\xi^{\beta\gamma})\in V_{\alpha^{-1}}$ (that is, $\alpha\gamma\overline{\awleinf}$). It follows from Remark \ref{rem:partialmapshomeomorphisms} that 
	\begin{eqnarray*}
		\phi_{\alpha\beta^{-1}}(\xi^{\beta\gamma})&=& \phi_{\alpha}(\phi_{\beta^{-1}}(\xi^{\beta\gamma}))  \\
		&=&G_{(\alpha)\gamma}(H_{[\beta]\gamma}(\xi^{\beta\gamma})). \\
	\end{eqnarray*}
	
`	(ii) If $t\notin \{\eword\}\cup\{\alpha\mid\alpha\in\awplus\}\cup\{\alpha^{-1}\mid\alpha\in\awplus\}\cup\{\beta\alpha^{-1}\mid \beta,\alpha\in\awplus, r(\alpha)\cap r(\beta)\neq\emptyset\}$, then $t$ either: (1) contains a factor of form $ab$ with $a,b\in\alf{}$ and $r(a)\cap s(b)=\emptyset$, (2) contains a factor of the form $ b^{-1}a^{-1}$ with $a,b\in\alf{}$ and $r(a)\cap s(b)=\emptyset$, (3) contains a factor of the form $ b^{-1}a$ with $a,b\in\alf{}$ and $a\neq b$, or (4) contains a factor of the form $\beta_2\alpha_2^{-1}$ with $t=\beta_1\beta_2(\alpha_1\alpha_2)^{-1}$ (in reduced form),  $\alpha_2,\beta_2\in \awplus$ and $r(\alpha_2)\cap r(\beta_2)=\emptyset$. In cases (1) and (2) we have $V_{ab}=\phi_a(V_b)$ and $V_{b^{-1}a^{-1}}=\phi_{a^{-1}}(V_{b^{-1}})$ are empty because $r(a)\cap s(b)\neq\emptyset$ if and only if $ab\in\awplus$. In   case (3) we have that $V_{(a^{-1}b)^{-1}}=\phi^{-1}_{b}(V_{a})\neq\emptyset$ if and only if $V_a\cap V_b\neq\emptyset$. However, if $a\neq b$ then $V_a\cap V_b=\emptyset$, and thus $V_{(a^{-1}b)^{-1}}=\emptyset$. In case (4), since $r(\alpha_2)\cap r(\beta_2)=\emptyset$, it follows from Lemma \ref{lem:intersectionproperties} that $V_{\beta_2\alpha_2^{-1}}=V_{(\alpha_2\beta_2^{-1})^{-1}}=\emptyset$. Hence $V_t=\emptyset$ in each of the four case.  
\end{proof}

We now show that the maps defined in (\ref{partialactionDFN1}), (\ref{partialactionDFN2}) and (\ref{e:toppartialactionextend}) define a partial action of $\F$ on $\ftight$. The proof is inspired by \cite{MR3539347}. 
\begin{proposition} \label{prop:partialtopaction}
	Let $\lspace$ be  a weakly left-resolving labelled space.  Then
	$$\Phi=(\{V_t\}_{t\in\F},\{\phi_t\}_{t\in\F})$$ 
	is a semi-saturated orthogonal partial action of $\F$ on $\ftight$.
\end{proposition}
\begin{proof}
	Fix $t\in \F$. If $t=\eword$, then 
	\begin{equation*}\label{e:partialactionPROOF0}
	V_\eword=\ftight \text{ and } \phi_{\eword}=\mathrm{id}_{\ftight}.  
	\end{equation*}
	If $t\neq\eword$ is in reduced form, then, by Lemma  \ref{lem:domainsofreducedwords}(ii), we may assume that  $t=\alpha\beta^{-1}$ with $\alpha,\beta\in\awstar$. Then  Equation (\ref{e:toppartialactionextend}) and Lemma \ref{lem:domainpart1} show that $V_t$ is open in $\ftight$ . To see that $\phi_t$ is a homeomorphism, note that if $\alpha=a_1\cdots a_n$ and $\beta=b_1\cdots b_m$, then by Equation  (\ref{e:toppartialactionextend}) 
	$$ \phi_t=\phi_{a_1}\circ\cdots\phi_{a_n}\circ \phi_{b_n}^{-1} \cdots \circ \phi_{b_1}^{-1}.$$
	By \cite[Proposition 4.8]{Gil3} the maps $\phi_{a} |_{\ftight_{(a)\gamma}}=G_{(a)\gamma}$ and $\phi_{a}^{-1} |_{\ftight_{a\gamma}}=H_{[a]\gamma}$ are homeomorphisms for every $a\in\alf$. Therefore, since $V_{a}=\sqcup \ftight_{a\gamma}$ is a disjoint union, it follows that $\phi_a$ is a homeomorphism for every $a\in\alf$. Hence Equations (\ref{e:toppartialactionextend}) now show that $\phi_t$ is a homeomorphism. 	
	
	Next we show that $\phi_s(V_{s^{-1}}\cap V_t)=V_{s}\cap V_{st}$ and $\phi_s(\phi_t(\xi))=\phi_{st}(\xi)$ for every $\xi \in V_{t^{-1}}\cap V_{(st)^{-1}}$. 
	Let $st\in\F$ be in reduced form (that is, each factor is a finite product of elements from $\alfg{A}$ and their inverses, or the empty word). Then using (\ref{e:toppartialactionextend}) we see that
	\begin{eqnarray}
	V_{st}=\phi_{st}(V_{(st)^{-1}})&=& \phi_{st}(\phi_{t^{-1}}(V_{s^{-1}})) \nonumber \\
	&=&  \phi_{s}(V_{s^{-1}} \cap V_{t}). \label{e:partialactionPROOF1}
	\end{eqnarray} 
	Now suppose that $s,t\in \F$ are in reduced form. Let $r,s_1,t_1\in\F$ such that $s=s_1r$, $t=r^{-1}t_1$ and $s_1t_1$ is the reduced form of $st$, where $s_1=\eword$ if $s=r$ and $t_1=\eword$ if $t=r^{-1}$. Then 
	\begin{eqnarray}
	V_{s^{-1}}\cap V_{t}&=& V_{(r^{-1}s_1^{-1})}\cap V_{r^{-1}t_1}\nonumber \\
	&=&\phi_{r^{-1}}(V_{r}\cap V_{s_1^{-1}})\cap \phi_{r^{-1}}(V_{r}\cap V_{t_1}) \text{\hspace{0.5cm} (by Equation (\ref{e:partialactionPROOF1}))} \nonumber \\
	&=& \phi_{r^{-1}}(V_{r}\cap V_{s_1^{-1}}\cap V_{t_1}) .\label{e:partialactionPROOF2}
	\end{eqnarray}
	Applying $\phi_s$ to the left and right hand sides of Equation (\ref{e:partialactionPROOF2}) we have that
	\begin{eqnarray}
	\phi_s(V_{s^{-1}}\cap V_{t})&=& \phi_{s_1r}(\phi_{r^{-1}}(V_{r}\cap V_{s_1^{-1}}\cap V_{t_1}))\nonumber \\
	&=& \phi_{s_1}(V_{s_1^{-1}}\cap V_{r}) \cap \phi_{s_1}(V_{s_1^{-1}}\cap V_{t_1})\nonumber \\
	&=&  V_{(r^{-1}s_1^{-1})^{-1}}\cap V_{s_1t_1}\text{ \hspace{0.5cm} (by Equations (\ref{e:toppartialactionextend}) and (\ref{e:partialactionPROOF1}))} \nonumber\\
	&=& V_{s}\cap V_{st}. \nonumber 
	\end{eqnarray}
	Hence 
	\begin{equation*}\label{e:partialactionPROOF3}
	\phi_s(V_{s^{-1}}\cap V_t)=V_{s}\cap V_{st}.
	\end{equation*}
	
	To show that $\phi_s\circ\phi_t=\phi_{st}$ on  $V_{t^{-1}}\cap V_{(st)^{-1}}$, first note that 
	$V_{t^{-1}}\cap V_{(st)^{-1}}=\phi_{t^{-1}}(V_{t})\cap\phi_{t^{-1}}(V_{s^{-1}}\cap V_t)= \phi_{t^{-1}}(V_{s^{-1}})$. Hence, $\phi_s\circ\phi_t$ and $\phi_{st}$ have the same domains. That 
	\begin{equation*}\label{e:partialactionPROOF4}
	\phi_s(\phi_t(\xi))=\phi_{st}(\xi)
	\end{equation*}
	for every $\xi \in V_{t^{-1}}\cap V_{(st)^{-1}}$ now follows from (\ref{e:toppartialactionextend}). 
	This completes the proof that $\Phi$ is a partial action of $\F$ on $\ftight$.
	
	It remains to show that the action is semi-saturated and orthogonal. Let $s=a_1\cdots a_n$ and $t=b_1\cdots b_m$ be elements of $\F$ in reduced form such that $\mid st \mid=\mid s\mid+ \mid t\mid$. Then the element $st$ is in reduced form and by Equation (\ref{e:partialactionPROOF1}) we have that 
	\begin{equation*}
	V_{st}= \phi_{s}(V_{s^{-1}} \cap V_{t})\subseteq V_s.
	\end{equation*}
	Hence the action is semi-saturated by \cite[Proposition 4.1]{MR1953065}. 
	If $\alpha\neq \beta$, then $V_\alpha\cap V_\beta =\emptyset$. Hence the action is orthogonal, completing the proof.
\end{proof}

\section{Labelled space $C^*$-algebra as a crossed product}\label{s:crossedprodisom}
In this section we show that the partial crossed product $C^*$-algebra obtained from the partial action in Section \ref{s:partialaction} is isomorphic to the labelled space $C^*$-algebra $C^*\lspace$. 

We begin by describing the partial crossed product $C^*$-algebra $C_0(\ftight)\rtimes_{\ph}\F$. Let $\lspace$ be a weakly left-resolving labelled space and let $\Phi=(\{V_t\}_{t\in\F},\{\phi_t\}_{t\in\F})$ be the partial action associated with $\lspace$ in Proposition \ref{prop:partialtopaction}. For every $\alpha\in\awstar$ and $A\in\acfrg{\alpha}$, the subset $V_{(\alpha,A,\alpha)}\subset \ftight$ is compact and open, with the exception of $V_{\eword}=\ftight$ which might not be compact. Hence any $f\in C_0(V_{(\alpha,A,\alpha)})$ can and will be viewed as a function in $C_0(\ftight)$ by declaring that $f(\xi)=0$ if $\xi\notin V_{(\alpha,A,\alpha)}$. In fact, $C_0(V_{(\alpha,A,\alpha)})$ is a closed two-sided ideal in $C_0(\ftight)$ and thus a $C^*$-subalgebra. In particular, this applies to 
the sets $\{V_t\}_{t\in\F\backslash \{\eword\}}$, by Lemma \ref{lem:domainsofreducedwords}.
Put
$$D_{t}=C_0(V_t) \text{ and }  D_{t^{-1}}=C_0(V_{t^{-1}}).$$ 
Define $\ph_t:D_{t^{-1}}\to D_t$  
by 
$$\ph_t(f)=f\circ\phi_{t^{-1}},$$
and define $\ph_{t^{-1}}:D_t\to D_{t^{-1}}$ analogously. Then $(\{D_t\}_{t\in\F},\{\ph_t\}_{t\in\F})$ is a $C^*$-algebraic partial dynamical system \cite{MR1953065}.  Hence we may consider the partial crossed product 
\begin{equation}\label{e:partcrossedprodDFN}
C_0(\ftight)\rtimes_{\ph} \F=\overline{\mathrm{span}}\left\{ \sum_{t\in\F}f_t\delta_t : f_t\in D_t \text{ and } f_t\neq0 \text{ for finite many  }t\in\F \right\}, 
\end{equation}
where the closure is with respect to the universal norm (see for example \cite[Definition 11.11]{MR3699795}). Note that $\delta_t$ has no meaning in itself and merely serves as a place holder.\footnote{ Alternatively if one views $f_t\delta_t$ as a function from $G$ into $C_0(\ftight)$, then $f_t\delta_t$ can be viewed as $f_t$-valued with support equal to $\{t\}$.}  
Recall that multiplication and involution in $C_0(\ftight)\rtimes_{\ph} \F$ are given by 
\begin{equation}\label{e:crossedprodOperations}
\begin{split}
&(a\delta_s) (b\delta_t)=\ph_s(\ph_{s^{-1}}(a)b)\delta_{st}, \text{ and}\\
&(a\delta_s)^{*}=\ph_{s^{-1}}(a)\delta_{s^{-1}}.\\
\end{split}
\end{equation}

The labelled space $C^*$-algebra $C^*\lspace$ is generated by a set of projections $\{p_A\mid A\in \acf \}$ and a set of partial isometries $\{s_a\mid a\in \alfg{A}\}$ subject to certain relations (see Section \ref{sec:labelCstarAlgDfn}). If $A\in\acf$ and $\alpha\in\mathcal{L}^{*}$, then we let $1_{V_{(\alpha,A,\alpha)}}$ denote the characteristic function on $V_{(\alpha,A,\alpha)}$. We show that $C_0(\ftight)$ is generated by these characteristic functions.

\begin{lemma}\label{lem:genC0} 
	The $C^*$-algebra $C_0(\ftight)$ is generated by the set 
	$$\{1_{V_{(\alpha,A,\alpha)}}\mid \alpha\in\awstar, A\in\acfrg{\alpha}\}.$$
\end{lemma}
\begin{proof}
	We employ the Stone-Weierstrass Theorem. 
	Let $\xi\in \ftight$. Since $\xi$ is filter it is non-empty. Suppose $p\in\xi$, with $p=(\gamma,A,\gamma)$ and $0\leq\mid\gamma\mid\leq \infty$. Then   $1_{V_{(\gamma,A,\gamma)}}(\xi)=1$. Hence the set $\{1_{V_{(\alpha,A,\alpha)}}\mid \alpha\in\awstar, A\in\acfrg{\alpha}\}$ vanishes nowhere. 
	
	Let $\xi,\eta\in\ftight$.  We show that $\{1_{V_{(\alpha,A,\alpha)}}\mid \alpha\in\awstar, A\in\acfrg{\alpha}\}$ separates points. Assume $\xi\neq\eta$. Without loss of generality we may assume that $(\gamma,A,\gamma)\in\xi$ and $(\gamma,A,\gamma)\notin\eta$. Then $\xi\in V_{(\gamma,A,\gamma)}$ and $\eta\notin V_{(\gamma,A,\gamma)}$. Hence $1_{V_{(\gamma,A,\gamma)}}(\xi)\neq1_{V_{(\gamma,A,\gamma)}}(\eta)$, which shows that the set $\{1_{V_{(\alpha,A,\alpha)}}\mid \alpha\in\awstar, A\in\acfrg{\alpha}\}$ separates points. The lemma now follows from the Stone-Weierstrass Theorem.

\end{proof}

Let $1_A$ denote the characteristic function on $V_{(w,A,w)}$ and $1_\alpha$  the characteristic function on $V_{(\alpha,r(\alpha),\alpha)}$.
In  Proposition \ref{prop:genCrossProd} we show that the $C^*$-subalgebra generated by $\{1_A\delta_w,1_a\delta_a\mid a\in\alf, A\in\acf\}$ and the crossed product $C^*$-algebra $C_0(\ftight)\rtimes_{\ph}\F$ are the same. This requires some computations with elements of the form $1_A\delta_w$ and $1_\alpha\delta_\alpha$. The purpose of the following lemma is to to de-clutter the proofs that follow, by provide a reference to some of these computations, and so making the proofs easier to parse.

\begin{lemma}\label{lem:crossProdComputations}
	Let $C^*(\{1_A\delta_w,1_a\delta_a\})\subset C_0(\ftight)\rtimes_{\ph}\F$ denote the $C^*$-subalgebra generated by $\{1_A\delta_w,1_a\delta_a\mid a\in\alf, A\in\acf\}$. Then
	\begin{enumerate}[(i)]
		\item $\ph_{\alpha^{-1}}(1_{V_{(\alpha,A,\alpha)}})=1_{V_{(\eword,A,\eword)}}$ and $\ph_\alpha(1_{V_{(\eword,A,\eword)}})=1_{V_{(\alpha,A,\alpha)}}$ for every $\alpha\in\awplus$ and $A\in\acf_{\alpha}$,
		\item  $1_{\alpha}\delta_\alpha=(1_{a_1}\delta_{a_1})\cdots(1_{a_n}\delta_{a_n})$ for all $\alpha\in\awplus$ where $\alpha=a_1a_2\cdots a_n$,
		\item\label{lem:crossProdComputations_triple}  $ (1_{\alpha}\delta_{\alpha})(1_A\delta_\eword)(1_{\alpha}\delta_{\alpha})^{*}=1_{V_{(\alpha,A,\alpha)}}\delta_\eword$ for all $\alpha\in\awplus$ and $A\in\acf_{\alpha}$,
		\item $(1_\alpha\delta_\alpha)^*=1_{r(\alpha)}\delta_{\alpha^{-1}}=1_{\alpha^{-1}}\delta_{\alpha^{-1}}$ for every $\alpha\in\awplus$,
		\item $1_{\alpha\beta^{-1}}\delta_{\alpha\beta^{-1}}=(1_{\alpha}\delta_{\alpha})(1_{\beta^{-1}}\delta_{\beta^{-1}})$ for every  $\alpha,\beta\in\awplus$ such that $\alpha\beta^{-1}\in \F$ is in reduced form.
		\item $ (1_A\delta_\eword)(1_a\delta_a)=(1_a\delta_a)(1_{r(A,a)}\delta_\eword)$ for all $\alpha\in\awplus$ and $A\in\acf$.
	\end{enumerate}
\end{lemma}
\begin{proof}
	(i) Let  $\alpha\in\awplus$ and $A\in\acf_{\alpha}$. We first show that \begin{equation}\label{e:indicatorshiftequal}
	\ph_{\alpha^{-1}}(1_{V_{(\alpha,A,\alpha)}})=1_{V_{(\eword,A,\eword)}}.
	\end{equation} 
	Note that $\ph_{\alpha^{-1}}(1_{V_{(\alpha,A,\alpha)}})=1_{V_{(\alpha,A,\alpha)}}\circ\phi_\alpha =1_{\phi_{\alpha^{-1}}(V_{(\alpha,A,\alpha)})}.$ Hence to show Equation 
	(\ref{e:indicatorshiftequal}) we need to show that $\phi_{\alpha^{-1}}(V_{(\alpha,A,\alpha)}) =V_{(\eword,A,\eword)}$. Note that 
	\begin{eqnarray*}
		\xi^{\alpha\beta}\in V_{(\alpha,A,\alpha)} &\Leftrightarrow&  (\alpha,A,\alpha)\in\xi^{\alpha\beta} \\
		&\Leftrightarrow& A\in\xi^{\alpha\beta}_{|\alpha|}. \\	
	\end{eqnarray*}
	Put $\eta=\phi_{\alpha^{-1}}(\xi^{\alpha\beta})$. Then 
	\begin{eqnarray*}
		A\in\xi^{\alpha\beta}_{|\alpha|}&\Leftrightarrow& A\in\eta_0 \\
		&\Leftrightarrow& (\eword,A,\eword)\in\eta \\
		&\Leftrightarrow& \eta=\phi_{\alpha^{-1}}(\xi^{\alpha\beta})\in V_{(\eword,A,\eword)}.\\
	\end{eqnarray*}
	%
	Hence  $\phi_{\alpha^{-1}}(V_{(\alpha,A,\alpha)})= V_{(\eword,A,\eword)}$, which shows that $\ph_{\alpha^{-1}}(1_{V_{(\alpha,A,\alpha)}})=1_{V_{(\eword,A,\eword)}}$. Since, $\ph_\alpha$ and $\ph_{\alpha^{-1}}$ are inverses of each other, it follows that 
	$\ph_\alpha(1_{V_{(\eword,A,\eword)}})=1_{V_{(\alpha,A,\alpha)}}$, completing the proof of (i).

	(ii) First let  $a\in\alf$ and $\beta\in\awstar$.	Then 
	\begin{equation*} 
	(1_a\delta_a)(1_\beta\delta_\beta)= \ph_a(\ph_{a^{-1}}(1_a)1_\beta)\delta_{a\beta}=\ph_a(1_{r(a)}1_\beta)\delta_{a\beta},  
	\end{equation*} 
	where $\ph_a(1_{r(a)}1_\beta)\in D_{a\beta}=C_0(V_{a\beta})$. Let $\xi\in V_{a\beta}=\phi_a(V_\beta).$ Then $\xi=\xi^{a\beta\gamma}$ for some $\gamma\in \overline{\awleinf}$. Thus  $\phi_{a^{-1}}(\xi^{a\beta\gamma})\in V_\beta\cap V_{a^{-1}}$ and 
	\begin{equation*}
	\begin{split}
	\ph_a(1_{r(a)}1_\beta)(\xi^{a\beta\gamma})& =1_{r(a)}1_\beta (\phi_{a^{-1}}(\xi^{a\beta\gamma}))\\
	& = 1_{r(a)}1_\beta \left(H_{[\alpha]}(\xi^{\alpha\beta\gamma})\right)\\
	& = 1.
	\end{split}
	\end{equation*} 
	Hence $\ph_a(1_{r(a)}1_\beta)=1_{a\beta}$, and thus $(1_a\delta_a)(1_\beta\delta_\beta)=1_{a\beta}\delta_{a\beta}$. For $\alpha=a_1 \cdots a_n$, applying the preceding argument pairwise from right to left yields $ (1_{a_1}\delta_{a_1})\cdots(1_{a_n}\delta_{a_n})=1_{a_1\cdots a_n}\delta_{a_1\cdots a_n}=1_{\alpha}\delta_\alpha$.
	
	(iii) Let $\alpha\in\awplus$ and $A\in\acf_{\alpha}$. Then applying Equations (\ref{e:crossedprodOperations}) and (i) we have that
	\begin{eqnarray*}
		(1_{\alpha}\delta_{\alpha})(1_A\delta_\eword)(1_{\alpha}\delta_{\alpha})^{*} &=&\left(\ph_\alpha\left(\ph_{\alpha^{-1}}(1_\alpha)1_A \right)\delta_\alpha\right)  \left(\ph_{\alpha^{-1}}(1_\alpha)\delta_{\alpha^{-1}}\right) \\
		&=& \left(\ph_\alpha\left(1_{r(\alpha)}1_A \right)\delta_\alpha\right)  \left(1_{r(\alpha)}\delta_{\alpha^{-1}}\right) \\
		&=& \ph_{\alpha}(1_{r(\alpha)}1_A1_{r(\alpha)})\delta_\eword  \\
		&=& \ph_{\alpha}(1_A)\delta_\eword  \\
		&=& 1_{V_{(\alpha,A,\alpha)}}\delta_\eword.
	\end{eqnarray*}
	
	(iv)  Let $\alpha\in\awplus$. By definition $(1_\alpha\delta_\alpha)^*=\ph_{\alpha^{-1}}(1_\alpha)\delta_{\alpha^{-1}}$. Since $\ph_{\alpha^{-1}}$ is a *-isomorphism from $C_0(V_{\alpha})$ onto $C_0(V_{\alpha^{-1}})$, it follows that $\ph_{\alpha^{-1}}(1_\alpha)=1_{\alpha^{-1}}$. On the other hand $\ph_{\alpha^{-1}}(C_0(V_{\alpha}))= C_0(\phi_{\alpha^{-1}}(V_{\alpha}))=C_0(V_{\alpha^{-1}})=C_0(V_{(\eword,r(\alpha),\eword)})$. Hence $\ph_{\alpha^{-1}}(1_\alpha)=1_{\alpha^{-1}}=1_{r(\alpha)}$.

	(v) By (iv) above we have that 
	$$(1_\alpha\delta_\alpha)(1_{\beta^{-1}}\delta_{\beta^{-1}})=(1_\alpha\delta_\alpha)(1_{r(\beta)}\delta_{\beta^{-1}})=
	\ph_\alpha(1_{r(\alpha)\cap r(\beta)})\delta_{\alpha\beta^{-1}}=(1_{r(\alpha)\cap r(\beta)}\circ\phi_{\alpha^{-1}})\delta_{\alpha\beta^{-1}}.$$ 
	Hence it suffices to show that $1_{\alpha\beta^{-1}}=1_{r(\alpha)\cap r(\beta)}\circ\phi_{\alpha^{-1}}$. By Equations (\ref{e:toppartialactionextend}) we have that 
	$$V_{\alpha\beta^{-1}}=\phi_{\alpha}(V_{\beta^{-1}})=\phi_{\alpha}(V_{\alpha^{-1}}\cap V_{\beta^{-1}}). $$  
	Thus, for $\xi\in\ftight$, we have that  
	\begin{eqnarray}
	1_{\alpha\beta^{-1}}(\xi)\neq 0 &\Leftrightarrow& \xi\in \phi_{\alpha}(V_{\alpha^{-1}}\cap V_{\beta^{-1}}) \nonumber\\
	&\Leftrightarrow& \phi_{\alpha^{-1}}(\xi)\in V_{\alpha^{-1}}\cap V_{\beta^{-1}}\nonumber\\ 
	&\Leftrightarrow& (\eword,r(\alpha),\eword),(\eword,r(\beta),\eword)\in  \phi_{\alpha^{-1}}(\xi)\nonumber\\
	&\Leftrightarrow& (\eword,r(\alpha)\cap r(\beta),\eword)\in\phi_{\alpha^{-1}}(\xi)\nonumber\\
	&\Leftrightarrow& 1_{r(\alpha)\cap r(\beta)}\circ\phi_{\alpha^{-1}}(\xi)\neq 0.\nonumber
	\end{eqnarray}
	Hence $1_{\alpha\beta^{-1}}=1_{r(\alpha)\cap r(\beta)}\circ\phi_{\alpha^{-1}}$,  completing the proof of (v).
	
	(vi) Note that 
	$$ (1_A\delta_\eword)(1_a\delta_a)=(1_A 1_a)\delta_a.$$ 
	We claim that $1_A 1_a=\ph_a(1_{r(a)}1_{r(A,a)})$.  For $\xi\in\ftight$ we have 
	\begin{eqnarray*}
		1_A1_a(\xi)\neq0 &\Leftrightarrow& \xi\in V_{(\eword,A,\eword)}\cap V_{(a,r(a),a)} \\
		&\Leftrightarrow& r(A,a)\cap r(a)\in\xi_1 \text{ and } \xi=\xi^{a\beta}, 
	\end{eqnarray*}
	for some $\beta \in\ftight$. On the other hand
	\begin{eqnarray*}
		\ph_a(1_{r(a)}1_{r(A,a)})(\xi)=1_{r(a)\cap r(A,a)}(\phi_{a^{-1}}(\xi))\neq 0&\Leftrightarrow& \phi_{a^{-1}}(\xi)\in V_{r(a)\cap r(A,a)}\\
		&\Leftrightarrow&  \xi\in \phi_a(V_{r(a)\cap r(A,a)}) \\
		&\Leftrightarrow&  \xi \in V_a \text{ and } A\in \xi_0  \\
		&\Leftrightarrow&  \xi=\xi^{a\beta}\text{ and } A\in \xi_0  \\
		&\Leftrightarrow&  r(A,a)\cap r(a)\in\xi_1 \text{ and } \xi=\xi^{a\beta}, 
	\end{eqnarray*}
	for some $\beta \in\ftight$, which proves the claim that  $1_A 1_a=\ph_a(1_{r(a)}1_{r(A,a)})$. Then 
	\begin{eqnarray*}
		(1_A\delta_\eword)(1_a\delta_a)&=& 1_A1_a\delta_a\\
		&=& \ph_a(1_{r(a)}1_{r(A,a)})\delta_a \\
		&=& \ph_{a}(\ph_{a^{-1}}(1_a)1_{r(A,a)})\delta_a \\
		&=&  (1_a\delta_a)(1_{r(A,a)}\delta_\eword).  \\
	\end{eqnarray*}
\end{proof}

The labelled space $C^*$-algebra $C^*\lspace$ is generated by a set of projections and a set of partial isometries indexed by sets $ A\in \acf $ and letters $ a\in \alfg{A}$, respectively (subject to certain relations (Section \ref{sec:labelCstarAlgDfn})). In the following proposition we show that  $C_0(\ftight)\rtimes_{\ph} \F$ also has generators indexed by the sets $ A\in \acf $ and letters $ a\in \alfg{A}$. This will help to define the isomorphism between these $C^*$-algebras.

\begin{proposition} \label{prop:genCrossProd}
	Let $C^*(\{1_A\delta_w,1_a\delta_a\})\subseteq C_0(\ftight)\rtimes_{\ph}\F$ denote the $C^*$-subalgebra generated by $\{1_A\delta_w,1_a\delta_a\mid a\in\alf, A\in\acf\}$. Then
	$$C_0(\ftight)\rtimes_{\ph}\F=C^*(\{1_A\delta_w,1_a\delta_a\}).$$
\end{proposition}
\begin{proof}
	Let $C_0(\ftight)\delta_\eword$ denote the canonical image of $C_0(\ftight)$ in $C_0(\ftight)\rtimes_{\ph}\F$. We claim that 
	$C_0(\ftight)\delta_\eword \subset C^*(\{1_A\delta_w,1_a\delta_a\})$. To prove  this claim, note that by Lemma \ref{lem:genC0} it will suffice to see that $1_{V_{(\alpha,A,\alpha)}}\delta_\eword\in C^*(\{1_A\delta_\eword,1_a\delta_a\})$ for every $\alpha\in\awstar$ and $A\in\acf$. 
	Hence to prove this claim let  $\alpha=a_1 \cdots a_n$. Then by  by Lemma \ref{lem:crossProdComputations} we have that
	$$ (1_{a_1}\delta_{a_1})\cdots(1_{a_n}\delta_{a_n})=1_{a_1\cdots a_n}\delta_{a_1\cdots a_n},$$ and also 
	$$ (1_{\alpha}\delta_{\alpha})(1_A\delta_\eword)(1_{\alpha}\delta_{\alpha})^{*}=1_{V_{(\alpha,A,\alpha)}}\delta_\eword.$$ 
	Hence 
	$$ 1_{V_{(\alpha,A,\alpha)}}\delta_\eword = (1_{a_1}\delta_{a_1}\cdots1_{a_n}\delta_{a_n})1_{A}\delta_\eword (1_{a_1}\delta_{a_1}\cdots1_{a_n}\delta_{a_n})^{*} $$
	proving our claim.
	
	Next we show that $C_0(\ftight)\rtimes_{\ph}\F=C^*(\{1_A\delta_w,1_a\delta_a\})$. It is clear that $C^*(\{1_A\delta_w,1_a\delta_a\})\subseteq C_0(\ftight)\rtimes_{\ph}\F$. To see the reverse inclusion, let $f_t\delta_t\in C_0(\ftight)\rtimes_{\ph}\F$ with $f_t\in C_0(\ftight)$ and let $t\in\F$. Then $f_t\in D_t=C_0(V_t)$. If $t=\eword$ then $f_\eword\delta_\eword\in C_0(\ftight)\delta_{\eword}\subset  C^*(\{1_A\delta_w,1_a\delta_a\})$. If $t\neq \eword$, then by Lemma \ref{lem:domainsofreducedwords}(ii) we may assume that $t=\alpha\beta^{-1}$ with $\alpha,\beta\in\awstar$. We consider three cases. First assume that $|\alpha|,|\beta|\geq 1$. Then by Lemma \ref{lem:crossProdComputations} we have that 
	\begin{eqnarray*}
		f_{\alpha\beta^{-1}}\delta_{\alpha\beta^{-1}} &=&  (f_{\alpha\beta^{-1}}1_{\alpha\beta^{-1}})\delta_{\alpha\beta^{-1}} \\
		&=& (f_{\alpha\beta^{-1}}\delta_\eword)(1_{\alpha}\delta_{\alpha})(1_{\beta^{-1}}\delta_{\beta^{-1}}) \\
		&=& (f_{\alpha\beta^{-1}}\delta_\eword)(1_{\alpha}\delta_{\alpha})(1_{\beta}\delta_{\beta})^{*}, 
	\end{eqnarray*}
	which shows that $f_{\alpha\beta^{-1}}\delta_{\alpha\beta^{-1}}\in C^*(\{1_A\delta_w,1_a\delta_a\})$. Secondly, assume that $\beta=\eword$ and $\alpha\neq \eword$. Then 
	\begin{eqnarray*}
		f_{\alpha}\delta_{\alpha} &=&  (f_{\alpha}1_{\alpha})\delta_{\alpha} \\
		&=& (f_{\alpha}\delta_\eword)(1_{\alpha}\delta_{\alpha}), 
	\end{eqnarray*}
	which shows that $f_{\alpha}\delta_{\alpha}\in C^*(\{1_A\delta_w,1_a\delta_a\})$. Lastly, assume that $\alpha=\eword$ and $\beta\neq \eword$. Then by Lemma \ref{lem:crossProdComputations}(iv) we have that 
	\begin{eqnarray*}
		f_{\beta^{-1}}\delta_{\beta^{-1}} 
		&=&  (f_{\beta^{-1}}1_{\beta^{-1}})\delta_{\beta^{-1}} \\
		&=&  (f_{\beta^{-1}}\delta_\eword)(1_{\beta^{-1}}\delta_{\beta^{-1}}) \\
		&=&(f_{\beta^{-1}}\delta_\eword)(1_{\beta}\delta_{\beta})^{*},
	\end{eqnarray*}
	which shows that $f_{\beta^{-1}}\delta_{\beta^{-1}} \in C^*(\{1_A\delta_w,1_a\delta_a\})$.
	Hence, a set that densely spans  $C_0(\ftight)\rtimes_{\ph}\F$ (see (\ref{e:partcrossedprodDFN})) is contained in  $C^*(\{1_A\delta_w,1_a\delta_a\})$, which shows that $C_0(\ftight)\rtimes_{\ph}\F\subseteq C^*(\{1_A\delta_w,1_a\delta_a\})$, and completes the proof.

\end{proof}

By \cite[Theorem 4.3]{MR1953065}, if $N:	\alf \to (0,\infty)$ is any function, then there exists a unique strongly continuous one-parameter group $\sigma$ of automorphisms of $ C_0(\ftight)\rtimes_{\ph}\F $ such that 
$$  \sigma_t(f\delta_a)=N(a)^{it}f\delta_a \hspace{1cm} \text{ and } \hspace{1cm} \sigma_t(g\delta_w)=g\delta_w, $$
for all $t\in\mathbb{R}, a\in\alf,f\in D_g$ and $g\in D_w$.	
If we let $N(a)=\mathrm{exp}(1)$ for every $a\in\alf$, then we obtain a strongly continuous action (using the same notation) $\sigma:\mathbb{T}\to \mathrm{Aut}(C_0(\ftight)\rtimes_{\ph}\F)$ such that 
\begin{equation}
\begin{split}
&\sigma_z(f\delta_a)=zf\delta_a, \text{ and } \\
&\sigma_z(g\delta_w)=f\delta_w,
\end{split}
\end{equation}
for all $t\in\mathbb{R},  a\in\alf,f\in D_g$ and $g\in D_w$.

We now prove our main result of this section.

\begin{theorem}\label{thm:IsomorphismThm}
	There is a *-isomorphism $\psi$ from $C^*\lspace$ onto  $ C_0(\ftight)\rtimes_{\ph}\F$ such that 
	$$\psi(p_A)=1_A\delta_\eword \hspace{1cm} \text{ and  } \hspace{1cm} \psi(s_a)= 1_a\delta_a.$$
\end{theorem}
\begin{proof}
	By Proposition \ref{prop:genCrossProd}, $C_0(\ftight)\rtimes_{\ph}\F$ is densely spanned by the set
	$$ \{1_A\delta_w,1_a\delta_a\mid A\in\acf, a\in\alf{} \}.$$
	We show this set satisfies the relations of $C^*\lspace$ as in Section \ref{sec:labelCstarAlgDfn}. Fix $a,b\in\alf$, $A,B\in\acf$ and let
	$\xi\in\ftight$ with complete family $\{\xi_n\}$. 
	
	For relation (i), note that $(1_A\delta_\eword)(1_B\delta_\eword)=(1_A1_B)\delta_\eword$ and $1_A1_B(\xi)\neq0$ if and only if $\xi\in V_{(\eword,A,\eword)}\cap V_{(\eword,B,\eword)}$. That is, 
	$A\cap B\in\xi_0$, which implies that $\xi\in V_{(\eword,A\cap B,\eword)}$. On the other hand if $\xi\in V_{(\eword,A\cap B,\eword)}$ then $A\cap B\in\xi_0$. Since $\xi_0$ is a filter and since $ A\cap B\subset A$ and $ A\cap B\subset B$ it follows that $A\in \xi_0$ and $B\in\xi_0$. Hence $\xi\in V_{(\eword,A,\eword)}\cap V_{(\eword,B,\eword)}$, and thus $(1_A\delta_\eword)(1_B\delta_\eword)= 1_{A\cap B}\delta_\eword$. If $A=\emptyset$, then $(\eword,A,\eword)\notin E(S)$ and thus $V_{(\eword,A,\eword)}=\{\xi\in \ftight\mid (\eword,A,\eword)\in\xi\}=\emptyset$. Hence $1_A\delta_\eword=0$. The equation $$1_{A\cup B}\delta_\eword=1_A\delta_\eword+1_B\delta_\eword-1_{A\cap B}\delta_\eword$$
	follows from the fact that $\xi_0$ is a prime filter. Therefore $A\cup B\in\xi_0$ if and only if $A\in\xi_0$ or $B\in\xi_0$, proving relation (i). 	
	
	Relation (ii) follows directly from Lemma \ref{lem:crossProdComputations}(vi).
	
	Next we show that relation (iii) is satisfied.  Firstly, we have that
	$$(1_a\delta_a)^{*}(1_a\delta_a)=(1_{r(a)}\delta_a^{-1})(1_a\delta_a)=1_{r(a)}\delta. $$
	Secondly, if  $a\neq b $  then $V_{a^{-1}b}=\emptyset $ by Lemma \ref{lem:domainsofreducedwords}(ii), and thus 
	$$ (1_a\delta_a)^{*}(1_b\delta_b)=0.$$
	
	To see that relation (iv) is satisfied, we need to show, for every $A\in\acf$ for which $0<\card{\lbf(A\dgraph^1)}<\infty$ and there does not exist $B\in\acf$ such that $\emptyset\neq B\subseteq A\cap \dgraph^0_{sink}$, that
	\begin{eqnarray}\label{e:IsomorphismThm1}
	1_A\delta_\eword&=&\sum_{a\in\lbf(A\dgraph^1)}(1_a\delta_a)(1_{r(A,a)}\delta_\eword)(1_a^*\delta_a)\nonumber \\
	&=& \sum_{a\in\lbf(A\dgraph^1)} \ph_a(1_{r(A,a)})\delta_\eword\nonumber \\
	&=& \sum_{a\in\lbf(A\dgraph^1)} (1_{V(\eword,r(A,a),\eword)}\circ\phi_{a}^{-1})\delta_\eword\nonumber \\
	&=& \sum_{a\in\lbf(A\dgraph^1)} (1_{\phi_{a}(V_{(\eword,r(A,a),\eword)})})\delta_\eword. 
	\end{eqnarray}
	Note that if $a,b\in \lbf(A\dgraph^1)$, then $\phi_{a}(V_{(\eword,r(A,a),\eword)}) \subset V_a$ and $\phi_{b}(V_{(\eword,r(A,b),\eword)}) \subset V_b$. Hence, since $V_a\cap V_b=\emptyset$ if $a\neq b$, it follows that $\phi_{a}(V_{(\eword,r(A,a),\eword)}) \cap \phi_{b}(V_{(\eword,r(A,b),\eword)})=\emptyset$. Therefore to show that Equation (\ref{e:IsomorphismThm1}) holds it will suffice to show that
	\begin{equation}
	V_{(\eword,A,\eword)}=\bigcup_{a\in\lbf(A\dgraph^1)} \phi_{a}(V_{(\eword,r(A,a),\eword)}).
	\end{equation}
	Let $\xi^{\alpha}\in V_{(\eword,A,\eword)}$ with $\alpha=a_1a_2\cdots$ (possibly the empty word). Then $A\in\xi_0^{\alpha}$. Since $0<\card{\lbf(A\dgraph^1)}<\infty$ and there does not exist $B\in\acf$ such that $\emptyset\neq B\subseteq A\cap \dgraph^0_{sink}$, it follows from \cite[Theorem 6.7]{MR3648984}  that $\alpha\neq \eword$, $a_1\in \lbf(A\dgraph^1)$ and $r(A,a_1)\neq\emptyset$. Hence $r(A,a_1)\in\xi_1^{\alpha}$. Thus 
	$$ \phi_{a_1}(\xi^{\alpha})\in V_{(\eword,r(A,a_1),\eword)}, $$
	and since $\xi^{\alpha}=\phi_{a_1^{-1}}(\phi_{a_1}(\xi^{\alpha}))$ it follow that 
	$$ V_{(\eword,A,\eword)}\subseteq \bigcup_{a\in\lbf(A\dgraph^1)} \phi_{a}(V_{(\eword,r(A,a),\eword)}). $$
	For the reverse inclusion let $\xi^{\beta}\in V_{(\eword,r(A,a),\eword)}$ for some $a\in\lbf(A\dgraph^1)$ and let $\eta=\phi_a(\xi^{\beta})$. Then $r(A,a)\in\xi_0^{\beta}$. Hence $r(A,a)\cap r(a)=r(A,a)\in \eta_1$, which implies that $A\in\eta_0$ and thus $\eta\in V(\eword,A,\eword)$, giving the reverse inclusion. Hence Equation \ref{e:IsomorphismThm1} is satisfied. 
	
	Since $C^*\lspace$ is universal for these relations, there is a surjective *-homomorphism $\psi:C^*\lspace\to C_0(\ftight)\rtimes_{\ph}\F$ such that 
	$$\psi(p_A)=1_A\delta_\eword \hspace{1cm} \text{ and  } \hspace{1cm} \psi(s_a)= 1_a\delta_a.$$
	
	All that is left to prove is that the $\psi$ is injective. By \cite[Crollary 3.10]{MR3614028} $\psi$ is injective if and only if $1_A\neq0$ for $A\neq \emptyset$, and for each $z\in\mathbb{T}$ there exists a *-homomorphism $\sigma_z:C^*(\{1_A,1_a\})\to C^*(\{1_A,1_a\})$ such that  
	\begin{equation}\label{e:stronglyctsgrhom}
	\begin{split}
	&\sigma_z(1_a\delta_\alpha)=z1_a\delta_\alpha, \text{ and } \\
	&\sigma_z(1_A\delta_w)=1_A\delta_w,
	\end{split}
	\end{equation}
	for all $A\in\acf$ and $a\in\alf$. Since the partial action is semi-saturated and orthogonal, it follows from \cite[Theorem 4.3]{MR1953065} and the discussion preceding this theorem that such a *-homomorphism $\sigma_z$ satisfying Equations (\ref{e:stronglyctsgrhom}) exists.  Hence $\psi$ is an isomorphism and the proof is complete. 
\end{proof}

\section{Partial action and labelled space groupoids are isomorphic}\label{sec:groupoids}
Let $\lspace$ be a normal labelled space. In this section we show that the groupoid associated with a partial action (as defined in \cite{MR2045419}) is isomorphic to the groupoid associated with a labelled space (as defined in \cite{Gil3}). 

For the definition of a groupoid we refer the reader to \cite{MR584266}. A topological groupoid is a groupoid with a topology such that  multiplication and  involution are continuous. 
A locally compact groupoid is \textit{\' etale }if the range and source maps are local homeomorphisms. An open set in a groupoid is a \textit{bisection} if the range and source maps are bijections when restricted to this open set. An \textit{ample} groupoid is an \' etale groupoid that has a base  of  compact open bisections.

We first describe the groupoid of $\lspace$ as in \cite{Gil3}.  The set
\begin{equation}\label{eqn:Groupoid1}
 \Gamma=\{(\xi^{\alpha\gamma},|\alpha|-|\beta|,\eta^{\beta\gamma})\in \ftight\times\mathbb{Z}\times\ftight\ |\ H_{[\alpha]\gamma}(\xi^{\alpha\gamma})=H_{[\beta]\gamma}(\eta^{\beta\gamma})\}
\end{equation}
is a groupoid with products and inverses given by
\[ (\xi,m,\eta)(\eta,n,\rho)=(\xi,m+n,\rho) \text{ and  } (\xi,m,\eta)^{-1}=(\eta,-m,\xi),\] 
respectively (\cite[Proposition 3.5]{Gil3}). We denote by $\Gamma^{(2)}$ the set of composible pairs. 
If $s=(\alpha,A,\beta)\in S\lspace$ and $e_1,\ldots,e_n\in E(S)$ then the sets 
\[ Z_{s,e:e_1,\ldots,e_n}=\{(\eta^{\alpha\gamma},|\alpha|-|\beta|,\xi^{\beta\gamma})\in\Gamma\ |\ \xi\in V_{e:e_1,\ldots,e_n}\text{ and }H_{[\alpha]\gamma}(\eta)=H_{[\beta]\gamma}(\xi) \} \]
form a basis of compact open sets for a locally compact Hausdorff topology on $\Gamma$ \cite[Proposition 4.4]{Gil3}. It follows from \cite[Corollary 4.14]{Gil3} that $\Gamma$ is an \' etale groupoid. Moreover,  $\Gamma$ is an ample groupoid since the topology has a basis of compact open sets. Note that the unit space $\Gamma^0$ of $\Gamma$ is identified with $\ftight$.  Let $r_{\Gamma}$ and $s_{\Gamma} $ denote the range and source maps, respectively, of $\Gamma$.  

Next we describe the groupoid of a partial action as in \cite{MR2045419}. Let $\Phi=(\{V_t\}_{t\in\F},\{\phi_t\}_{t\in\F})$ be the partial action on $\ftight$ as in Proposition \ref{prop:partialtopaction}. Then
\begin{equation}\label{eqn:Groupoid2}
 \mathcal{G}=\{(\xi,t,\eta)\in\ftight\times\F\times\ftight \mid \eta\in V_{t^{-1}}, \text{ and } \xi=\phi_t(\eta)\}
 \end{equation}
is a groupoid with products and inverses given by 
\[ (\xi,s,\eta)(\eta,t,\rho)=(\xi,st,\rho) \text{ and }  (\xi,t,\eta)^{-1}=(\eta,t^{-1},\xi),\] 
respectively \cite{MR2045419}. We give $\mathcal{G}$ the topology inherited from the product topology on $\ftight\times\F\times\ftight$. The unit space $\mathcal{G}^0$ is also identified with $\ftight$. Let $r_{\mathcal{G}}$ and $s_{\mathcal{G}} $ denote the range and source maps, respectively, of $\mathcal{G}$.  

Our main goal is to show that the groupoids $\Gamma$ and $\mathcal{G}$ are isomorphic (Theorem \ref{t:isomgroupoids}). To this end we first show that $\mathcal{G}$ is an ample groupoid. The following identification of elements in $\mathcal{G}$ is used throughout this section without reference.

\begin{lemma}\label{lem:groupoidsamesets}
	We have that  $(\xi,t,\eta)\in\mathcal{G}$ if and only if $(\xi,t,\eta)= (\xi^{\alpha\gamma},\alpha\beta^{-1},\eta^{\beta\gamma})$ for some $\alpha,\beta\in\awstar$, and $ \phi_{\beta^{-1}}(\eta^{\beta\gamma})=\phi_{\alpha^{-1}}(\xi^{\alpha\gamma}) $.
\end{lemma}
\begin{proof}
	Let $(\xi^{\alpha\gamma},\alpha\beta^{-1},\eta^{\beta\gamma})$ with $\alpha,\beta\in\awstar$, and such that  $\phi_{\beta^{-1}}(\eta^{\beta\gamma})=\phi_{\alpha^{-1}}(\xi^{\alpha\gamma})$. Then 
	$\eta\in V_{\alpha\beta^{-1}}\subseteq V_{\beta}$ (Remark \ref{rem:partialmapshomeomorphisms}) and  $\xi=\phi_{\alpha\beta^{-1}}(\eta)$ (since $\phi_{\alpha^{-1}}$ is homeomorphism). Hence $(\xi^{\alpha\gamma},\alpha\beta^{-1},\eta^{\beta\gamma})\in\mathcal{G}$.
	
	If $(\xi,t,\eta)\in\mathcal{G}$, then $\eta\in V_{t^{-1}}$ and thus $V_t\neq \emptyset$. Hence $t=\alpha\beta^{-1}$ for some $\alpha,\beta\in\awstar$ by Lemma \ref{lem:domainsofreducedwords}(ii), and $V_{t^{-1}}\subseteq V_\beta$ by Remark \ref{rem:partialmapshomeomorphisms}. That is, $\eta=\eta^{\beta\gamma}$ for some $\gamma\in\awleinf$. 
	Since $(\xi,t,\eta)\in\mathcal{G}$  we have that $\xi=\phi_{\alpha\beta^{-1}}(\eta^{\beta\gamma})$. Then,  Lemma \ref{lem:domainsofreducedwords}(i), it follows  that $\xi=\xi^{\alpha\gamma}$ and 
	$$ \phi_{\beta^{-1}}(\eta^{\beta\gamma})=\phi_{\alpha^{-1}}(\xi^{\alpha\gamma}). $$
\end{proof}

\begin{lemma} \label{lem:amplegroupoid}
	The groupoid $\mathcal{G}$ is an ample groupoid.
\end{lemma}
\begin{proof}
	
	Since $\F$ is discrete and $\ftight$ has a basis of compact open subsets, it follows that $\mathcal{G}$ has a basis of compact open sets. Hence we only need to show that $\mathcal{G}$ is \' etale. 
	Fix any $(\xi^{\alpha\gamma},\alpha\beta^{-1},\eta^{\beta\gamma})\in \mathcal{G}$. To show that $\mathcal{G}$ is an \' etale groupoid, we need to find a neighborhood $U$ of $(\xi^{\alpha\gamma},\alpha\beta^{-1},\eta^{\beta\gamma})$ such that the range map $r_{\mathcal{G}}$ of $\mathcal{G}$ restricted to $U$ is a homeomorphism. Let $W=\phi_{\alpha^{-1}}(V_\alpha)\cap (\phi_{\beta^{-1}}(V_\beta))$, and put $W_\alpha=\phi_\alpha(W)$ and $W_\beta=\phi_\beta(W)$. Then  
	$$ U= (W_\alpha\times\{\alpha\beta^{-1}\}\times W_\beta)\cap\mathcal{G}$$
	is an open neighborhood of $(\xi^{\alpha\gamma},\alpha\beta^{-1},\eta^{\beta\gamma})$ such that $r_{\mathcal{G}}|_{U}$ is just the projection onto the first coordinate,  and is thus a homeomorphism. Hence $\mathcal{G}$ is \' etale  and thus also an ample groupoid.
\end{proof}

We now prove our main result of this section
\begin{theorem}\label{t:isomgroupoids}
	The map $\Theta:\mathcal{G}\to \Gamma$ defined by 
	$$\Theta((\xi^{\alpha\gamma},\alpha\beta^{-1},\eta^{\beta\gamma}))=(\xi^{\alpha\gamma},|\alpha|-|\beta|,\eta^{\beta\gamma}) $$
	is a groupoid isomorphism.
\end{theorem}
\begin{proof}	 
	We begin by showing that $\Theta$ is well-defined. Suppose $t\in\F$ and $t=\alpha\beta$ where $\alpha=\alpha^{\prime}\gamma$ and $\beta=\gamma^{-1}\beta^{\prime}$ such that $t=\alpha^{\prime}\beta^{\prime}$ is in reduced form. Then $|\alpha|-|\beta|=|\alpha^{\prime}\gamma|-|\gamma^{-1}\beta^{\prime}|=|\alpha^{\prime}|-|\beta^{\prime}|$. Therefore, since $\Theta$ only changes the second coordinate, it follows that $\Theta$ is well-defined.
	
	We show that $\Theta$ is a bijection. It is clear that $\Theta$ is onto. Let $(\xi^{\alpha\gamma},\alpha\beta^{-1},\eta^{\beta\gamma})$ and $(\zeta^{\mu\delta},\mu\nu^{-1},\chi^{\nu\delta})$ be in $\mathcal{G}$. 
	Note that if $\Theta((\xi^{\alpha\gamma},\alpha\beta^{-1},\eta^{\beta\gamma}))=\Theta((\zeta^{\mu\delta},\mu\nu^{-1},\chi^{\nu\delta}))$, then 
	$\xi^{\alpha\gamma}=\zeta^{\mu\delta}$ and $\eta^{\beta\gamma}=\chi^{\nu\delta}$. Hence $\alpha\gamma=\mu\delta$ and $\beta\gamma=\nu\delta$ by Theorem \ref{thm.filters.in.E(S)}. We may assume $\alpha=\mu\alpha^{\prime}$ (if not, then $\mu=\alpha\mu^{\prime}$ and the same argument holds). Hence $\alpha\gamma=\mu\alpha^{\prime}\gamma=\mu\delta$, implying that $\alpha^{\prime}\gamma=\delta$. Then $\beta\gamma=\nu\alpha^{\prime}\gamma$, which implies that $\beta=\nu\alpha^{\prime}$. Then we compute in $\F$:
	$$ \alpha\beta^{-1}= \mu\alpha^{\prime}(\nu\alpha^{\prime})^{-1}=\mu\nu^{-1}.$$ 
	Hence $\Theta$ is a bijection.   
	
	It is straightforward to show that $\Theta$ preserves multiplication and inverses.
	
	It remains to show that $\Theta$ is also a homeomorphism. However, since both $\Gamma$ and $\mathcal{G}$ are ample groupoids, it suffices to show that their unit spaces are homeomorphic. However this follows from the fact that both $\Gamma^{(0)}$ and $\mathcal{G}^{(0)}$ are homeomorphic to $\ftight$, completing the proof.
\end{proof}

\section{Simplicity of labelled spaces C*-algebras}\label{section:simplicity}
In this section we characterize simplicity of $C^*\lspace$ in terms of the tight spectrum $\ftight$. We do this by using known simplicity results for groupoid $C^*$-algebras and partial crossed product C*-algebras.  

\begin{definition}
	Let $\mathcal{G}$ be a locally compact, Hausdorff groupoid. A unit $u\in \mathcal{G}^{(0)}$ has \emph{trivial isotropy} if the set $\mathcal{G}^u_u=\{\gamma\in G\mid s(\gamma)=r(\gamma)=u\}$ contains only $u$. A subset $D\scj \mathcal{G}^{(0)}$ is \emph{invariant} if for all $\gamma\in \mathcal{G}$, when $s(\gamma)\in D$ then $r(\gamma)\in D$. We say that $\mathcal{G}$ is \emph{topologically principal} if the set of units with trivial isotropy is dense in $\mathcal{G}^{(0)}$, and \emph{minimal} if $\mathcal{G}^{(0)}$ and $\emptyset$ are the only open invariant subsets of $\mathcal{G}^{(0)}$.
\end{definition}


If a locally compact Hausdorff groupoid $\mathcal{G}$ is minimal and has a unit with trivial isotropy, then $\mathcal{G}$ is  topologically principal \cite[Remark 2.2]{MR3189105}.

\begin{definition}
	Let $\Phi=(\{U_t\}_{t\in G},\{\phi_t\}_{t\in G})$ be a partial action of a discrete group $G$ on a locally compact Hausdorff space $X$. Then $\Phi$ is \emph{topologically free} if the set of fixed points $\text{Fix}(t)=\{x\in X\mid x\in U_{t^{-1}}\text{ and }\phi_t(x)=x\}$ has empty interior, for all $t\in G\setminus\{e\}$. A subset $V\scj X$ is \emph{invariant} under the partial action if $\phi_t(V\cap U_{t^{-1}})\scj V$ for all $t\in G$. The partial action is \emph{minimal} if $\emptyset$ and $X$ are the only open invariant subsets of $X$.
\end{definition}

It is not hard to see that if $\Phi=(\{U_t\}_{t\in G},\{\phi_t\}_{t\in G})$ is a partial action on $X$, then the following are equivalent:
	\begin{enumerate}[(1)]
		\item $\Phi$ is minimal,
		\item for every $x\in X$ the \emph{orbit} of $x$, $\text{Orb}(x):=\{\phi_t(x)\mid t\in G,x\in U_{t^{-1}}\}$, is dense in $X$,
		\item for every non-empty subset $V$ of $X$ the \emph{orbit} of $V$, $\text{Orb}(V):=\bigcup_{t\in G} \phi_t(V\cap U_{t^{-1}})$, is equal to $X$.
	\end{enumerate}

By comparing the above definitions, one arrives at:

\begin{proposition}\label{prop:conditions.groupoid.partial.action.equivalent}
	Let $\Phi=(\{U_t\}_{t\in G},\{\phi_t\}_{t\in G})$ be a partial action of a discrete group $G$ on $X$ and let $\mathcal{G}_\Phi$ be its associated groupoid. Then:
	\begin{enumerate}
		\item $\mathcal{G}_\Phi$ is topologically principal if and only if $\Phi$ is topologically free, and
		\item $\mathcal{G}_\Phi$ is minimal if and only if $\Phi$ is minimal.
	\end{enumerate}
\end{proposition}

In following two theorems we state known simplicity characterizations for groupoid C*-algebras and partial crossed products.

\begin{theorem}\cite[Theorem 5.1]{MR3189105} \label{thm.simple.groupoid}
	Let $\mathcal{G}$ be a second-countable, locally compact and Hausdorff  \'etale groupoid. Then the groupoid C*-algebra $C^*(\mathcal{G})$ is simple if and only if the following conditions are satisfied:
	\begin{enumerate}
		\item $C^*(\mathcal{G})=C^*_r(\mathcal{G})$,
		\item $\mathcal{G}$ is topologically principal,
		\item $\mathcal{G}$ is minimal.
	\end{enumerate}
\end{theorem}

\begin{theorem}\cite[Corollary 2.9]{MR1905819}\label{thm.simple.partial.action}
	Let $\Phi=(\{U_t\}_{t\in G},\{\phi_t\}_{t\in G})$ be a partial action of a discrete group $G$ on a locally compact Hausdorff space $X$. If $\Phi$ is topologically free and minimal, then the associated partial crossed product is simple.
\end{theorem}

Note that, since $G$ is a discrete group, the groupoid $\mathcal{G}_\Phi$ associated with the partial action in Theorem \ref{thm.simple.partial.action} is an \'etale groupoid (see for example the proof of Lemma \ref{lem:amplegroupoid}) . Hence, if $\mathcal{G}_\Phi$ is second-countable and $C^*(\mathcal{G}_\Phi)=C^*_r(\mathcal{G}_\Phi)$, then the converse of Theorem  \ref{thm.simple.partial.action} also holds (by Theorem \ref{thm.simple.groupoid}). For a normal labelled spaces and its associated partial action $\hat{\Phi}$, it is always the case that $C^*(\mathcal{G}_{\hat{\Phi}})=C^*_r(\mathcal{G}_{\hat{\Phi}})$, \cite[Corollary 4.12]{Gil3}. However, since we do not assume that the underlying graph is countable, $\mathcal{G}_{\hat{\Phi}}$ may not be second-countable (although it is in many examples).


With the relationship between a partial action $\Phi$ and its associated groupoid $\mathcal{G}_\Phi$ given in Proposition \ref{prop:conditions.groupoid.partial.action.equivalent}, we can describe being topologically principle and minimal in terms of a family of open sets in  $X$ that satisfy a certain property. Specifically, suppose that $\mathfrak{F}$ is a family of open sets of $\mathcal{G}_\Phi^{(0)}=X$ with the property that for every non-empty open subset $U$ of $X$ there exists a non-empty set $V\in\mathfrak{F}$ such that $V\scj U$. Then, to check that $\mathcal{G}$ is topologically principal, it is sufficient to show that for every $V\in\mathfrak{F}$, there exists $u\in V$ such that $u$ has trivial isotropy. To check that $\mathcal{G}_\Phi$ is minimal, it is sufficient to show that for all $x\in X$ and all $V\in\mathfrak{F}$, we have that $\text{Orb}(x)\cap V\neq \emptyset$.

We show that a normal labelled space has such a family $\mathfrak{F}$ with which we can characterize simplicity of a labelled space $C^*$-algebra solely in terms of the tight spectrum of the labelled space. For the remainder of this section fix a labelled space $\lspace$ with tight spectrum $\ftight$. Let $\mathfrak{F}=\{V_e\subset\ftight \mid e\in E(S)\}$. Then, even though $\mathfrak{F}$ is not necessarily a basis for $\ftight$, the following lemma shows that it does satisfy the property described in the previous paragraph.

\begin{lemma}\label{lemma:refinement.ftight}
Let $U$ be a non-empty open subset of $\ftight$. Then there exists $e\in E(S)$ such that $\emptyset\neq V_e\scj U$.
\end{lemma}

\begin{proof}
Since the set of ultrafilters is dense in $\ftight$, there exists an ultrafilter $\xi\in U$. By \cite[Proposition 2.5]{MR3448414}, the family $\{V_e\mid e\in\xi\}$ is a neighbourhood basis for $\xi$, and therefore there exists $e\in\xi$ such that $\emptyset\neq V_e\scj U$.
\end{proof}

Applying Lemma \ref{lemma:refinement.ftight} and using Theorem \ref{thm.simple.groupoid} and Theorem \ref{thm.simple.partial.action}, we state a simplicity characterization for labelled space $C^*$-algebras:

\begin{theorem}\label{thm:simplicity.tight.spectrum}
	Let $\lspace$ be a normal labelled space, $\ftight$ the tight spectrum of the associated inverse semigroup, $\Phi$ the partial action given in Section \ref{s:partialaction}, and $\mathcal{G}$ the groupoid given in Section \ref{sec:groupoids}. If
	\begin{itemize}
		\item there exists $\xi\in\ftight$ with trivial isotropy and,
		\item for all $e\in E(S)$ and all $\xi\in\ftight$, there exists $\alpha,\beta\in\awstar$ such that $\xi\in V_{\alpha\beta^{-1}}$ and $\phi_{\beta\alpha^{-1}}(\xi)\in V_{e}$,
	\end{itemize}
	then $C^*\lspace$ is simple. Under the extra hypothesis that $\mathcal{G}$ is second-countable, the converse is also true, and moreover, for every non-empty open set $U\scj \ftight$, there exists $\xi$ with trivial isotropy such that $\xi\in U$.
\end{theorem}

We want to find conditions for simplicity solely on terms of the labelled space. We begin by characterizing when an element of $\ftight$ has non-trivial isotropy. As is shown in Section \ref{sec:groupoids}, we have two (isomorphic) groupoid models for $C^*\lspace$; one in terms of the partial action (Equation (\ref{eqn:Groupoid2})) and the other in terms the cutting map $H$ (Equation (\ref{eqn:Groupoid1})). By \cite[Proposition 4.8]{Gil3} the cutting map may also be seen as a shift $\sigma$, analogous to the case of directed graphs. We use these descriptions interchangeably without much fuss, if  no confusion is likely. We need the following technical lemmas.

\begin{lemma}\label{lemma:filter.loop.1}
	Let $\lspace$ be a normal labelled space. Suppose that $\ft\in X_{\beta\gamma}$ for some $\beta,\gamma\in\awstar$. If $A\cap r(A,\gamma)\neq\emptyset$ for all $A\in\ft$, then $A\cap r(A,\gamma)\in\ft$ for all $A\in\ft$.
\end{lemma}

\begin{proof}
Given $A\in\ft$, since $A\in\acf_{\beta\gamma}$, it follows that $A\cap r(A,\gamma)\in\acf_{\beta\gamma}$. To prove that $A\cap r(A,\gamma)\in\ft$, it is sufficient to show that $A\cap r(A,\gamma) \cap B\neq\emptyset$ for all $B\in\ft$. Since relative ranges preserve inclusion, if $B\in\ft$, then 
	\[A\cap r(A,\gamma)\cap B\supseteq (A\cap B)\cap r(A\cap B,\gamma),\]
	which is non empty by hypothesis since $A\cap B\in\ft$.
\end{proof}

For $\alpha,\beta\in\awplus$ such that $\alpha\beta\in\awplus$, let  $f_{\alpha[\beta]}$ and $h_{(\alpha)\beta}$ be the maps defined in Equation (\ref{eq:def.f}) and Equation (\ref{eq:def.h}), respectively (see Section \ref{subsection.filter.surgery}).
\begin{lemma}\label{lemma:filter.loop.2}
Let $\lspace$ be a normal labelled space. Suppose that $\ft\in X_{\beta\gamma^{n+1}}$ for some $\beta\in\awstar$ and $\gamma\in\awplus$ and $n\in\nn^*$. Then $h_{[\beta]\gamma^n}(f_{\beta\gamma^n[\gamma]}(\ft))=h_{[\beta\gamma]\gamma^n}(\ft)$ if and only if $A\cap r(A,\gamma)\neq\emptyset$ for all $A\in\ft$.
\end{lemma}

\begin{proof}
By the definition of $f$ and $h$, and by \cite[Lemma 4.7]{MR3680957}, we have that $h_{[\beta\gamma]\gamma^n}(\ft)=\usetr{\ft}{\acf_{\gamma^n}}$ and \[h_{[\beta]\gamma^n}(f_{\beta\gamma^n[\gamma]}(\ft))=f_{\gamma^n[\gamma]}(h_{[\beta]\gamma^{n+1}}(\ft))=\{D\in\acf_{\gamma^n}\mid r(D,\gamma)\in\usetr{\ft}{\acf_{\gamma^{n+1}}}\}.\]
	
	If $h_{[\beta]\gamma^n}(f_{\beta\gamma^n[\gamma]}(\ft))=h_{[\beta\gamma]\gamma^n}(\ft)$, then for $A\in\ft$, we have that $A\in h_{[\beta\gamma]\gamma^n}(\ft)$. Hence $r(A,\gamma)\in \usetr{\ft}{\acf_{\gamma^{n+1}}}$, which implies that $A\cap r(A,\gamma)\neq\emptyset$.
	
	Suppose now that $A\cap r(A,\gamma)\neq\emptyset$ for all $A\in\ft$ and let $D\in h_{[\beta\gamma]\gamma^n}(\ft)$. Then, $D\in\acf_{\gamma^n}$ and there exists $C\in \ft$ such that $C\scj D$. By Lemma \ref{lemma:filter.loop.1}, $C':=C\cap r(C,\gamma)\in\ft$, so that $C'\scj r(C,\gamma)\scj r(D,\gamma)$. This implies that $D\in h_{[\beta]\gamma^n}(f_{\beta\gamma^n[\gamma]}(\ft))$ and therefore $h_{[\beta\gamma]\gamma^n}(\ft)\scj h_{[\beta]\gamma^n}(f_{\beta\gamma^n[\gamma]}(\ft))$. However, since we are dealing with ultrafilters, these sets must be equal.
\end{proof}

For $\alpha,\beta\in\awplus$ such that $\alpha\beta\in\awplus$, let  $\sigma$ be the shift map given by $\sigma^{|\alpha|}=H_{\alpha[\beta]}$ (see \cite[Proposition 4.8]{Gil3}). 
\begin{proposition}\label{prop:non.trivial.isotropy}
	Let $\lspace$ be a normal labelled space. A tight filter $\xia$ has non trivial isotropy if and only if there exist $\beta,\gamma\in\awplus$ such that $\alpha=\beta\gamma^{\infty}$ and for all $n\in\nn^*$ and all $A\in\xi_{|\beta\gamma^n|}$ we have that $A\cap r(A,\gamma)\neq\emptyset$. 
\end{proposition}
\begin{proof}
	Suppose first that $\xia$ has non trivial isotropy. This means that there exist $k,l\in\nn$ with $k>l$ such that $\sigma^l(\xi)=\sigma^k(\xi)=:\eta$. By looking at the labelled path of $\eta$, we conclude that there exists $\gamma\in\awplus$ with $|\gamma|=k-l$, and such that the labelled path of $\eta$ is $\gamma^{\infty}$. Then there exists $\beta\in\awstar$ such that $|\beta|=l$ and $\alpha=\beta\gamma^{\infty}$.
	
	Let $n\in\nn$ and $A\in\xi_{|\beta\gamma^{n+1}|}$. Then
	\[h_{[\beta\gamma]\gamma^n}(\xi_{|\beta\gamma^{n+1}|})=\eta_{|\gamma^n|}=h_{[\beta]\gamma^n}(\xi_{|\beta\gamma^{n}|})=h_{[\beta]\gamma^n}(f_{\beta\gamma^n[\gamma]}(\xi_{|\beta\gamma^{n+1}|})).\]
	By Lemma \ref{lemma:filter.loop.2}, $A\cap r(A,\gamma)\neq\emptyset$.
	
	Suppose now that there exist $\beta,\gamma\in\awplus$ such that $\alpha=\beta\gamma^{\infty}$ and for all $n\in\nn^*$ and all $A\in\xi_{|\beta\gamma^n|}$ we have that $A\cap r(A,\gamma)\neq\emptyset$. Define $\eta=\sigma^{|\beta|}(\xi)$ and $\rho=\sigma^{|\beta\gamma|}(\xi)$. In order to prove that $\xi$ has non trivial isotropy, it is sufficient to show that $\eta=\rho$. Since a filter $\eta_m$ for $m\in\nn^*$ determines all filters $\eta_n$ for $n<m$, we can reduce the proof to showing that $\eta_m=\rho_m$ infinitely many times. Fix an arbitrary $n\in\nn^*$. Applying  Lemma \ref{lemma:filter.loop.2}, we see that
 \[\eta_{|\gamma^n|}=h_{[\beta]\gamma^n}(\xi_{|\beta\gamma^n|})=h_{[\beta]\gamma^n}(f_{\beta\gamma^n[\gamma]}(\xi_{|\beta\gamma^{n+1}|}))=h_{[\beta\gamma]\gamma^n}(\xi_{|\beta\gamma^{n+1}|})=\rho_{|\gamma^n|}.\]
The result now follows.
\end{proof}

\begin{remark}\label{remark:fixed.point} 
	An alternative approach to finding points with non trivial isotropy is to look for fixed points of the partial action $\Phi$ of Section \ref{s:partialaction}. Suppose $\xia\in \text{Fix}(t)$ for some $t\in\F\setminus\{\eword\}$. By the definition of the partial action, we must have $t=\mu\nu^{-1}$ for some $\mu,\nu\in\awstar$ such that $|\mu|\neq|\nu|$, and in this case $\sigma^{|\mu|}(\xi)=\sigma^{|\nu|}(\xi)$. For $\beta$ and $\gamma$ as in Proposition \ref{prop:non.trivial.isotropy}, we conclude from the proof of Proposition \ref{prop:non.trivial.isotropy} that $\mu=\beta$ and $\nu=\beta\gamma$ if $|\mu|<|\nu|$, and $\mu=\beta\gamma$ and $\nu=\beta$ if $|\mu|>|\nu|$. Here too we have that $\alpha=\beta\gamma^{\infty}$.
\end{remark}

\begin{definition}\cite[Definition 9.5]{MR3606190}
	Let $\lspace$ be a normal labelled space.
	\begin{enumerate}
		\item A pair $(\alpha,A)$ with $\alpha\in\awplus$ and $A\in\acfra$ is a \emph{cycle} if for every $B\in\acfra$ with $B\scj A$, we have that $r(B,\alpha)=B$.
		\item A cycle $(\alpha, A)$ has an \textit{exit} if there exists $0\leq k \leq |\alpha|$ and $\emptyset\neq B\in \acf$ such that $B\scj r(A,\alpha_{1,k})$ and $\lbf(B\dgraph^1)\neq \{\alpha_{k+1}\}$ (where $\alpha_{|\alpha|+1}:=\alpha_1$).
		\item The labelled space $\lspace$ satisfies condition ($L_\acf$) if every cycle has an exit.
	\end{enumerate}
\end{definition}

\begin{theorem}\label{thm:top.principal.labelled.spaces}
	Let $\lspace$ be a normal labelled space and $\mathcal{G}$ the groupoid given in Section \ref{sec:groupoids}. Then $\mathcal{G}$ is topologically principal if and only if $\lspace$ satisfies condition ($L_\acf$).
\end{theorem}

\begin{proof}
	First suppose that $\mathcal{G}$ is topologically principal and let $(\alpha,A)$ be a cycle. Consider the set $V_e$, where $e=(\alpha,A,\alpha)$, which is non-empty by Lemma \ref{lemma:ultrafilter}, and take $\xi=\xi^{\beta}\in V_e$ with trivial isotropy. We have a few cases to consider.
	
	If $\beta=\alpha^n\alpha_{1,k}\gamma$ for some $n\in\nn$, $0\leq k< |\alpha|$ and $\gamma\in\awleinf\setminus\{\eword\}$ such that $\gamma_1\neq \alpha_{k+1}$, then $\lbf(r(A,\alpha_{1,k})\dgraph^1)\supseteq \{\alpha_{k+1},\gamma_1\}\neq \{\alpha_{k+1}\}$, and hence $(\alpha,A)$ has an exit.
	
	If $\beta=\alpha^n\alpha_{1,k}$ for some $n\in\nn$ and $0\leq k< |\alpha|$, then $\xi$ is a tight filter of finite type, and by Theorem \ref{thm.tight.filters.in.es}, either $\lbf(r(A,\alpha_{1,k})\dgraph^1)$ is infinite (and therefore different from $\{\alpha_{t+1}\}$), or there exists $\emptyset\neq B\in\acf$ such that $B\scj r(A,\alpha_{1,k})\cap \dgraph^0_{sink}$, and for this $B$ we have that $\lbf(B\dgraph^1)=\emptyset\neq \{\alpha_{k+1}\}$. In both cases, $(\alpha,A)$ has an exit.
	
	If $\beta=\alpha^{\infty}$, by Proposition \ref{prop:non.trivial.isotropy}, there exist $n\in\nn^*$ and $B\in\xi_{|\alpha^n|}$ such that $B\cap r(B,\alpha)=\emptyset$. Since $(\alpha,A,\alpha)\in\xi$, we have that $A\in\xi_{|\alpha|}$, and since $A=r(A,\alpha)$, we have that $A=r(A,\alpha^{n-1})\in\xi_{|\alpha^n|}$. If we take $C=A\cap B\in \xi_{|\alpha^n|}\scj \acf$, then $\emptyset\neq C\scj A$ and $C\cap r(C,\alpha)=\emptyset$. This would imply that $C\neq r(C,\alpha)$, contradicting the fact that $(\alpha,A)$ is a cycle.
	
	For the converse, we prove the contra positive. Hence, assume that $\mathcal{G}$ is not topologically principal, or equivalently, that the partial action $\Phi$ given in Section \ref{s:partialaction} is not topologically free. Let $t\in\F\setminus\{\eword\}$ be such that $\text{Fix}(t)$ has non-empty interior. By Remark \ref{remark:fixed.point}, there exists $\beta,\gamma\in\awstar$ such that $t=\beta\gamma^{\pm 1}\beta^{-1}$. Now, by Lemma \ref{lemma:refinement.ftight}, there exists $e=(\delta,D,\delta)\in E(S)$ such that $\emptyset\neq V_e\scj \text{Fix}(t)$. We may assume without loss of generality that $\beta=\delta\beta'$ for some $\beta'\in\awstar$, since $t=\beta\gamma^n\gamma^{\pm 1}(\beta\gamma^n)^{-1}$ for any $n\in\nn$. For $A:=r(D,\beta'\gamma)$, we have that $A\scj r(\gamma)$. We claim that $(\gamma,A)$ is a cycle that has no exits, so that $\lspace$ does not satisfy condition ($L_\acf$).
	
	Since $V_e$ is non-empty, there exists $\xia\in V_e\scj\text{Fix}(t)$. By Remark \ref{remark:fixed.point}, $\alpha=\beta\gamma^{\infty}$ so that $A\in\xi_{|\beta\gamma|}$. By Proposition \ref{prop:non.trivial.isotropy}, $A\cap r(A,\gamma)\neq\emptyset$. 	If $(\gamma,A)$ is not a cycle, then either $A\setminus r(A,\gamma)\neq\emptyset$ or $r(A,\alpha)\setminus A\neq\emptyset$. Supposing that $B:=A\setminus r(A,\gamma)\neq\emptyset$, then
	\[B\cap r(B,\gamma)=(A\setminus r(A,\gamma))\cap(r(A,\gamma)\setminus r(A,\gamma^2))=\emptyset.\]
	For $f=(\beta\gamma,B,\beta\gamma)$ and $\eta\in\ftight$ such that $f\in\eta$, we have that $\eta\in V_e\scj\text{Fix}(t)$ since $f\leq e$. As with $\xi$, the associated labelled path of $\eta$ must be $\beta\gamma^{\infty}$. By Proposition \ref{prop:non.trivial.isotropy}, $\eta$ has trivial isotropy. But, this implies that $\eta\notin \text{Fix}(t)$, which is a contradiction. Analogously, if $B:=r(A,\gamma)\setminus A\neq\emptyset$, we can make the same argument with $f'=(\beta\gamma^2,B,\beta\gamma^2)$ to get a contradiction. This implies that $(\gamma,A)$ is a cycle.
	
	Assume now that $(\gamma,A)$ has an exit, that is there there exists $0\leq k \leq |\gamma|$ and $\emptyset\neq B\in \acf$ such that $B\scj r(A,\gamma_{1,k})$ and $\lbf(B\dgraph^1)\neq \{\alpha_{k+1}\}$. If $\lbf(B\dgraph^1)=\emptyset$, then a tight filter containing $(\beta\gamma\gamma_{1,k},B,\beta\gamma\gamma_{1,k})$ would be a filter of finite type belonging to $\text{Fix}(t)$. If $a\in \lbf(B\dgraph^1)\setminus\{\gamma_{k+1}\}$ we would get a tight filter whose associated labelled path begins with $\beta\gamma\gamma_{1,k}a$ belonging to $\text{Fix}(t)$. Both cases are contradictions because the labelled path of every tight filter in $\text{Fix}(t)$ is $\beta\gamma^{\infty}$. We conclude that $(\gamma,A)$ is a cycle with no exits, showing that  $\lspace$ does not satisfy condition ($L_\acf$).
\end{proof}

For a condition equivalent to minimality, but in terms of the labelled space, we need the notion of hereditary and saturated subsets of $\acf$.

\begin{definition}
	Let $\lspace$ be a normal labelled space. A subset $H$ of $\acf$ is  \emph{hereditary} if the following conditions hold:
	\begin{enumerate}[(i)]
		\item $r(A,\alpha)\in H$ for all $A\in H$ and all $\alpha\in\awstar$,
		\item $A\cup B\in H$ for all $A,B\in H$,
		\item if $B\in\acf$ is such that $B\scj A$ for some $A\in H$, then $B\in H$.
	\end{enumerate}
	A hereditary set $H$ is \emph{saturated} if given $A\in\acf_{reg}$ such that $r(A,a)\in H$ for all $a\in\alf$, then $A\in H$.
\end{definition}

Given a subset $\mathcal{I}$ of $\acf$, it is easy to see there is a smallest hereditary saturated set containing $\mathcal{I}$. In order to describe this set more concretely, we define
\[\her(\mathcal{I}):=\{B\in\acf\mid B\scj \bigcup_{i=1}^m r(A_i,\alpha_i)\text{ for some }m\in\nn,\ A_i\in\mathcal{I}\text{ and }\alpha_i\in\awstar,\ i=1,\ldots,m\},\]
and
\[\sat(\her(\mathcal{I}))=\bigcup_{n=0}^{\infty}\sat^{[n]}(\her(\mathcal{I})),\]
where $\sat^{[n]}(\her(\mathcal{I}))$ is defined inductively by
\begin{enumerate}[(i)]
	\item $\sat^{[0]}(\her(\mathcal{I}))=\her(\mathcal{I})$,
	\item for $n>0$, $\sat^{[n]}(\her(\mathcal{I}))=\{B\in\acf_{reg}\mid r(B,a)\in \sat^{[n-1]}(\her(\mathcal{I}))\text{ for all }a\in\alf\}$.
\end{enumerate}
The set $\sat(\her(\mathcal{I}))$ is then the smallest hereditary saturated set containing $\mathcal{I}$. In particular, if $\mathcal{I}=\{A\}$ for some $A\in\acf$, then we write $\sat(\her(A))$ for $\sat(\her(\mathcal{I}))$.

We now adapt a part of the proof from \cite[Theorem 9.15]{MR3606190} to our context.
\begin{theorem}\label{thm:minimality.labelled.spaces}
	Let $\lspace$ be a normal labelled space and $\mathcal{G}$ the groupoid given in Section \ref{sec:groupoids}. Then $\mathcal{G}$ is minimal if and only if $\{\emptyset\}$ and $\acf$ are the only hereditary saturated subsets of $\acf$.
\end{theorem}

\begin{proof}
	First suppose that $\mathcal{G}$ is minimal and take $A\in\acf\setminus\{\emptyset\}$ arbitrary. It is sufficient to show that $\sat(\her(A))=\acf$ and for that take $B\in\acf$. Since $\emptyset\in\sat(\her(A))$ holds for any $A$, we can suppose without loss of generality that $B\neq\emptyset$. Since $A\neq\emptyset$, for $e=(\eword,A,\eword)$, $V_e\neq\emptyset$ and therefore the orbit of $V_e$, namely, $\text{Orb}(V_e)=\bigcup_{\alpha,\beta\in\awstar}\phi_{\alpha\beta^{-1}}(V_e\cap V_{\beta\alpha^{-1}})$ is equal to $\ftight$. Since $B\neq\emptyset$, it follows that $V_{(\eword,B,\eword)}$ is a non-empty compact subset of $\ftight$ and therefore there exist $\alpha^1,\ldots,\alpha^m,\beta^1,\ldots,\beta^m\in\awstar$ such that $V_{(\eword,B,\eword)}\scj\bigcup_{i=1}^m\phi_{\alpha^i(\beta^i)^{-1}}(V_e\cap V_{\beta^i(\alpha^i)^{-1}})$. We use an induction argument on $N:=\max\{|\alpha_i|:i=1,\ldots,m\}$ to prove that $B\in \sat(\her(A))$. More precisely, we prove that for all $N\in\nn$, if $B\in\acf$ is such that there exist $\alpha^1,\ldots,\alpha^m,\beta^1,\ldots,\beta^m\in\awstar$ with $V_{(\eword,B,\eword)}\scj\bigcup_{i=1}^m\phi_{\alpha^i(\beta^i)^{-1}}(V_e\cap V_{\beta^i(\alpha^i)^{-1}})$ and $\max\{|\alpha_i|:i=1,\ldots,m\}=N$, then $B\in \sat(\her(A))$.
	
	If $N=0$, then the condition on $B$ is $V_{(\eword,B,\eword)}\scj\bigcup_{i=1}^m\phi_{(\beta^i)^{-1}}(V_e\cap V_{\beta^i})$. In this case, if $\xi\in V_{(\eword,B,\eword)}$, then there exists $i_0\in\{1,\ldots,m\}$ such that $\xi\in \phi_{(\beta^{i_0})^{-1}}(V_e\cap V_{\beta^{i_0}})$. Thus, there exists $\eta=\eta^{\beta^{i_0}\gamma}$ such that $A\in\eta_0$ and $\xi = H_{[\beta^{i_0}]\gamma}(\eta)$. By the definition of $H_{[\beta^{i_0}]\gamma}$, this implies that $r(A,\beta^{i_0})\in\xi_0$. We claim that $B\scj \bigcup_{i=1}^m r(A,\beta_i)$ so that $B\in\her(A)$. If not, then $C:=B\setminus\left(\bigcup_{i=1}^m r(A,\beta_i)\right)\in\acf$ and $C\neq\emptyset$. If $\xi\in\ftight$ such that $(\eword,C,\eword)\in\xi$ (which exists by Lemma \ref{lemma:ultrafilter}), then  $\xi\in V_{(\eword,B,\eword)}$ but $r(A,\beta^i)\cap C=\emptyset$, implying that $r(A,\beta^i)\notin\xi_0$ for all $i=1,\ldots,m$, which is a contradiction.
	
	Fix $N\in\nn$ and suppose that $V_{(\eword,B,\eword)}\scj\bigcup_{i=1}^m\phi_{\alpha^i(\beta^i)^{-1}}(V_e\cap V_{\beta^i(\alpha^i)^{-1}})$ is such that $\max\{|\alpha_i|:i=1,\ldots,m\}=N+1$. Define $C:=B\setminus\left(\bigcup_{i:\alpha^i=\eword}r(A,\beta^i)\right)$. We prove that $C\in\acf_{reg}$. If there exists $D\in\acf$ such that $\emptyset\neq D\scj C\cap \dgraph^0_{sink}$, then an ultrafilter $\ft$ in $\acf$ containing $D$ would be such that $\xi=\{(\eword,E,\eword)\mid E\in\ft\}$ is a tight filter in $V_{(\eword,B,\eword)}$ but not in $\bigcup_{i=1}^m\phi_{\alpha^i(\beta^i)^{-1}}(V_e\cap V_{\beta^i(\alpha^i)^{-1}})$, which is a contradiction. If $\lbf(C\dgraph^1)$ is infinite, then there exists $a\in \lbf(C\dgraph^1)$ such that $a\neq\alpha_1^i$ for all $i$ such that $|\alpha_i|\geq 1$. If $\xi$ is a tight filter containing $(a,r(C,a),a)$, then $\xi\in V_{(\eword,B,\eword)}$, but with $\xi\notin \bigcup_{i=1}^m\phi_{\alpha^i(\beta^i)^{-1}}(V_e\cap V_{\beta^i(\alpha^i)^{-1}})$, which again is a contradiction.  We conclude that $C\in\acf_{reg}$ and hence $V_{(\eword,C,\eword)}=\bigcup_{a\in\lbf(C\dgraph^1)}V_{(a,r(C,a),a)}$. For each $a\in\lbf(C\dgraph^1)$, we have that $(a,r(C,a),a)\leq (\eword,B,\eword)$, implying that 
	$$V_{(a,r(C,a),a)}\scj \bigcup_{i:\alpha^i_1=a}\phi_{\alpha^i(\beta^i)^{-1}}(V_e\cap V_{\beta^i(\alpha^i)^{-1}}).$$  
By applying $\phi_{a^{-1}}$ to the inclusion above, we get $V_{(\eword,r(C,a),\eword)}\scj \bigcup_{i:\alpha^i_1=a}\phi_{\alpha^i_{2,|\alpha^i|}(\beta^i)^{-1}}(V_e\cap V_{\beta^i(\alpha^i_{2,|\alpha^i|})^{-1}})$. Using the induction hypothesis $r(C,a)\in \sat(\her(A))$ for all $a\in \lbf(C\dgraph^1)$. Since $\sat(\her(A))$ is saturated and $C\in\acf_{reg}$, we conclude that $C\in\sat(\her(A))$ and the same is true for $B\cap C$. Finally, notice that $B\setminus C\scj \bigcup_{i:\alpha^i=\eword}r(A,\beta^i)$ so that $B\setminus C\in \sat(\her(A))$. Using that $\sat(\her(A))$ is closed under finite unions, we get that $B=(B\setminus C)\cup (B\cap C)\in \sat(\her(A))$. Hence $\{\emptyset\}$ and $\acf$ are the only hereditary saturated subsets of $\acf$. 
	
	For the converse, assume that the only hereditary saturated subsets of $\acf$ are $\{\emptyset\}$ and $\acf$. We want to prove that for any $e=(\gamma,C,\gamma)\in E(S)$ and any $\xi=\xia\in\ftight$, there is an element in the orbit of $\xi$ and in $V_e$. Note that since $C\neq\emptyset$, it follows that $\sat(\her(C))=\acf$. Consider first the case where $\alpha$ is infinite and let $A\in\xi_1$ be arbitrary. Then there exists $n\in\nn$ such that $A\in\sat^{[n]}(\her(\mathcal{C}))$, and therefore there exist $\beta_1,\ldots,\beta_n\in\awstar$ such that 
	\[r(A,\alpha_{2,n+1})\scj\bigcup_{i=1}^m r(C,\beta_i).\]
	By taking intersections and using that $\xi_{n+1}$ is a prime filter, we see that there exists $i_0\in\{1,\ldots,m\}$ such that $r(A,\alpha_{2,n+1})\cap r(C,\beta_{i_0})\in \xi_{n+1}$. To simplify  notation, we write $\beta$ for $\beta_{i_0}$. Since $C\scj r(\gamma)$, it follows that $r(A,\alpha_{2,n+1})\cap r(C,\beta)\scj r(\gamma\beta)$. Hence $r(\gamma\beta)\in h_{[\alpha_{1,n+1}]}(\xi_{n+1})$, implying that $H_{[\alpha_{1,n+1}]\alpha_{n+2,\infty}}(\xi)\in \ftight_{(\gamma\beta)\alpha_{n+2,\infty}}$. Also, observe that $r(C,\beta)\in H_{[\alpha_{1,n+1}]\alpha_{n+2,\infty}}(\xi)_0$, which implies that $(\gamma,C,\gamma)\in G_{(\gamma\beta)\alpha_{n+2,\infty}}(H_{[\alpha_{1,n+1}]\alpha_{n+2,\infty}}(\xi))=\phi_{\gamma\beta\alpha_{1,n+1}^{-1}}(\xi)$, and this element is in the orbit of $\xi$ and in $V_e$. In the case where $\alpha$ is finite, we take an arbitrary $A\in\xi_{|\alpha|}$, which is not an element of $\acf_{reg}$, and therefore $A\in\her(C)$. Repeating the same argument as in the infinite case with $A$ in the place of $r(A,\alpha_{2,n+1})$, we see that there exists $\beta\in\awstar$ such that $\phi_{\gamma\beta\alpha^{-1}}(\xi)\in V_e$, which completes the proof.	
\end{proof}

Combining the results of this section we get the following characterization for the simplicity of $C^*\lspace$ solely in terms of $\ftight$.

\begin{theorem}\label{theorem:simple.labelled.space.algebra}
	Let $\lspace$ be a normal labelled space. If the labelled space satisfies condition $(L_\acf)$, and $\{\emptyset\}$ and $\acf$ are the only hereditary saturated subsets of $\acf$, then $C^*\lspace$ is simple. If, in addition, the groupoid $\mathcal{G}$ is second countable, then the converse is also true.
\end{theorem}


%

\section{Cuntz-Pimsner algebras associated to subshifts}\label{sec:subshift}
In \cite{MR2380472} (see also \cite{carlsen2004operator}) a $C^*$-algebra is associated  with a one-sided subshift. These algebras were studied as a groupoid C*-algebras and as partial crossed products in \cite{carlsen2004operator,MR3639522,MR3527546,MR2680816}.  There is a topological space that serves both as the unit space of the groupoid and the space on which the free group acts.  In \cite[Example 4]{MR3614028}  it is shown that the $C^*$-algebra of a one-sided subshift may be realized as the $C^*$-algebra of a normal labelled space.  In this section we describe how the topological space referred to above, arises as the spectrum of a commutative unital C*-algebra, and  use this to study the $C^*$-algebra of a one-sided subshift in the language of labelled spaces by applying  the theory developed in \cite{Gil3} and in this paper.


We recall the necessary definitions (see also \cite{MR1369092}). Given a finite set $\alf$ viewed as a discrete topological space, the \emph{one-sided full shift space} is $\alf^{\nn}$ with the product topology. The map $\sigma:\alf^{\nn}\to\alf^{\nn}$ given by $\sigma(x_0,x_1,x_2\ldots)=(x_1,x_2,\ldots)$ is called the \emph{shift map}. A \emph{subshift}, also called a \emph{shift space}, is a non-empty subset $\ssf\scj \alf^{\nn}$ that is closed and invariant by the shift map in the sense that $\sigma(\ssf)\scj \ssf$. The \emph{language} of $\ssf$ is the set $L_{\ssf}=\{\alpha\in\alf^*\mid \alpha y\in\ssf\text{ for some }y\in\ssf\}$. We consider the empty word $\eword$ as an element of $L_{\ssf}$. Since $\ssf$ is closed in $\alf^{\nn}$, if $x\in \alf^{\nn}$ is such that $x_0\cdots x_n \in L_{\ssf}$ for all $n\in\nn$, then $x\in \ssf$. Also, without loss of generality, we may assume that $\alf\scj L_{\ssf}$. A common way of defining a subshift is to specify a set of \emph{forbidden words} $\mathcal{F}\scj\alf$ and setting $\ssf_{\mathcal{F}}$ to be the subset of $\alf^{\nn}$ of all elements $x\in\alf^{\nn}$ such that $x_{i,j}\notin \mathcal{F}$ for all $i,j\in\nn$. It can be shown that $\ssf_{\mathcal{F}}$ is indeed a subshift.

Given $\alpha,\beta\in L_{\ssf}$, we define the set $C(\alpha,\beta)=\{\beta x\in\ssf\mid\alpha x\in\ssf\}$. In particular, $Z_{\beta}:=C(\eword,\beta)$ is called a \emph{cylinder set}, $F_{\alpha}:=C(\alpha,\eword)$ is called a \emph{follower set} and $\ssf=C(\eword,\eword)$. It is well known that $\{Z_{\beta}\}_{\beta\in L_{\ssf}}$ is a basis of compact open sets for the product topology on $\ssf$ and that $F_{\alpha}$ is closed with respect to this topology. It follows immediately that $C(\alpha,\beta)=\sigma^{-|\beta|}(F_{\alpha})\cap Z_{\beta}$ is closed in $\ssf$. 
The commutative $C^*$-algebra mentioned in the first paragraph of this section, denoted by $\mathcal{D}_{\ssf}$, is defined as the C*-subalgebra of $\ell^{\infty}(\ssf)$ generated by the characteristic functions $1_{C(\alpha,\beta)}$ for all $\alpha,\beta\in L_{\ssf}$.

As in \cite{MR3614028}, we define the labelled space $(\dgraph_{\ssf}, \lbf_{\ssf}, \acf_{\ssf})$ associated with a subshift $\ssf$ as follows: The graph $\dgraph_{\ssf}$ is given by $\dgraph_{\ssf}^0=\ssf$, $\dgraph_{\ssf}^1=\{(x,a,y)\in\ssf\times\alf\times\ssf\mid x=ay\}$, $s(x,a,y)=x$ and $r(x,a,y)=y$. The labelling map is given by $\lbf_{\ssf}(x,a,y)=a$ and the accommodating family $\acf_{\ssf}$ is the Boolean algebra generated by the sets $C(\alpha,\beta)$ for $\alpha,\beta\in\lbf_{\ssf}$. Then the triple $(\dgraph_{\ssf}, \lbf_{\ssf}, \acf_{\ssf})$ is  a normal labelled space \cite[Lemma 5.5]{MR3614028}. By \cite[Proposition 5.7]{MR3614028}, the Cuntz-Pimsner algebra $C^*(\ssf)$ associated to $\ssf$, as defined in \cite{MR2380472}, is isomorphic to $C^*(\dgraph_{\ssf}, \lbf_{\ssf}, \acf_{\ssf})$ in such way that $\mathcal{D}_{\ssf}\cong\overline{\text{span}}\{p_A\mid A\in\acf_{\ssf}\}$.

For the remainder we fix a finite set $\alf$ and a subshift $\ssf\scj \alf^{\nn}$. In order to simplify the notation,  we will omit the sub-index $\ssf$ in $(\dgraph_{\ssf}, \lbf_{\ssf}, \acf_{\ssf})$. Observe that given $x\in\ssf$, we have that $x=s(x,x_0,\sigma(x))$ so that the graph has no sinks. Also, since $\alf$ is finite, for all $A\in\acf$ the set $\lbf(A\dgraph^1)$ is finite, which implies that the elements of the corresponding tight spectrum $\ftight$ are only of infinite type (Theorem \ref{thm:TightFiltersType}). Our first goal is to show that $\ftight$ is homeomorphic to the Stone dual of $\acf$.

Observe that the finite and infinite paths on the graph $\dgraph$ must be of the form
\[(x,x_0,\sigma(x))(\sigma(x),x_1,\sigma^2(x))\cdots(\sigma^n(x),x_n,\sigma^{n+1}(x))(\cdots),\]
for some $x\in \ssf$. Hence, $\awstar = L_{\ssf}$ and $\awinf = \ssf$. Also, for $\alpha\in\awstar$, we have that $r(\alpha)=C(\alpha,\eword)=F_{\alpha}$. 

We recall how to evaluate the relative range in this example, \cite[Equation (6)]{MR3614028}. For $\alpha,\beta\in L_{\ssf}$ such that $\beta\neq\eword$ and $a\in\alf$,
\[r(C(\alpha,\beta),a)=\begin{cases}
C(a,\eword)\cap C(\alpha,\beta_2\ldots\beta_{|\beta|}) & \text{if } \beta=a\beta_2,\ldots,\beta_{|\beta|}, \\
\emptyset & \text{otherwise}.
\end{cases}
\]
Also,
\[r(C(\alpha,\eword),a)=\begin{cases}
C(\alpha a,\eword) & \text{if }\alpha a\in L_{\ssf}, \\
\emptyset & \text{if }\alpha a\notin L_{\ssf}.
\end{cases}\]
More generally, $r(A,\alpha)=\{x\in\ssf\mid \alpha x\in A\}$.

Also, recall that $\acfrg{\alpha}$ is the set of  elements from $\acf$ which are contained in $r(\alpha)$ (Section \ref{subsect.labelled.spaces}),  and $X_{\alpha}$ is its Stone dual. The topology on $X_{\alpha}$ is given by the basic open sets $U_A=\{\ft\in X_{\alpha}\mid A\in\ft\}$, where $A\in\acfra$. In particular, for the empty word $\acf_{\eword}=\acf$. In Proposition \ref{prop:tightStoneDualHomeo} below, we prove that $\ftight$ is homeomorphic to $X_{\eword}$. To this end, we need the following lemmas. 

\begin{lemma}\label{lemma:rel.range.shift.onto}
	Let $\alpha,\beta\in L_{\ssf}$ be such that $\alpha\beta\in L_{\ssf}$. Then the map $r(\cdot,\beta):\acfra\to\acfrg{\alpha\beta}$ is onto.
\end{lemma}

\begin{proof}
	Since $r(\cdot,\beta)$ is a Boolean algebra homomorphism, it is sufficient to prove that the generators of $\acfrg{\alpha\beta}$ are in the range of $r(\cdot,\beta)$. Observe that $\acfrg{\alpha\beta}$ is generated by the sets $C(\alpha\beta,\eword)\cap C(\gamma,\delta)$ for $\beta,\gamma,\delta\in L_{\ssf}$. For this intersection not to be empty, it is necessary that $\alpha\beta\delta\in L_{\ssf}$, and in which case $C(\alpha\beta,\eword)\cap C(\gamma,\delta)=C(\alpha\beta,\eword)\cap C(\beta,\eword)\cap C(\gamma,\delta)=r(C(\alpha,\eword)\cap C(\gamma,\beta\delta),\beta)$. The result now follows since $C(\alpha,\eword)\cap C(\gamma,\beta\delta)\in \acfra$.
\end{proof}

\begin{lemma}\label{lemma:f.almost.onto}
	Let $\alpha,\beta\in L_{\ssf}$ be such that $\alpha\beta\in L_{\ssf}$. For all $\ft\in X_{\alpha}$ such that $F_{\alpha}\cap Z_{\beta}\in\ft$, there exists a unique $\ftg{G}\in X_{\alpha\beta}$ such that $f_{\alpha[\beta]}(\ftg{G})=\ft$.
\end{lemma}

\begin{proof}
	The result is trivial for $\beta=\eword$ since $f_{\alpha[\eword]}$ is the identity on $X_{\alpha}$. Hence, suppose that $\beta\neq\eword$ and let $\ft\in X_{\alpha}$ be such that $F_{\alpha}\cap Z_{\beta}\in\ft$. Consider $\ftg{G}:=\{r(A,\beta)\in\acfrg{\alpha\beta}\mid A\in\ft\}$. We prove that $\ftg{G}\in X_{\alpha\beta}$.
	
	For $A\in\ft$, since $A\cap F_{\alpha}\cap Z_{\beta}\neq \emptyset$, there exists $x\in\ssf$ such that $\alpha\beta x\in \ssf$ and $\beta x \in A$. Observe that $x\in r(A,\beta)$ so that $r(A,\beta)\neq\emptyset$.
	
	That $\ftg{G}$ is closed under intersections follows immediately from the fact that the relative range preserves intersections, since $(\dgraph_{\ssf}, \lbf_{\ssf}, \acf_{\ssf})$ is weakly left-resolving.
	
	Given $D\in \acfrg{\alpha\beta}$ such that $r(A,\beta)\scj D$, we have to show that $D\in\ftg{G}$. By Lemma \ref{lemma:rel.range.shift.onto}, there exists $B\in \acfra$ such that $r(B,\beta)=D$. Suppose, by way of contradiction, that $B\notin \ft$. Then, since $\ft$ is an ultrafilter, there is $C\in \ft$ such that $B\cap C=\emptyset$. In this case, $A\cap C\in \ft$ and, using the first part of this proof, we have that
	\[\emptyset\neq r(A\cap C,\beta)=r(A,\beta)\cap r(C,\beta)\scj r(B,\beta)\cap r(C,\beta)=r(B\cap C,\beta)=\emptyset,\]
	which is a contradiction.  Hence, $D\in\ftg{G}$ and that $\ftg{G}$ is a filter.
	
	To prove that $\ftg{G}$ is an ultrafilter in $\acfrg{\alpha\beta}$, suppose that there exists a filter $\ftg{H}$ in $\acfrg{\alpha\beta}$ such that $\ftg{G}\subsetneq\ftg{H}$. Arguing as in the previous paragraph, if $D\in\ftg{H}\setminus\ftg{G}$, there exists $C\in\ft$ such that $r(C,\beta)\cap D=\emptyset$, which would imply that $\ftg{H}$ is not a filter. Hence $\ftg{G}\in X_{\alpha\beta}$.
	
	Now, by the definition of $f_{\alpha[\beta]}$, it is clear that $f_{\alpha[\beta]}(\ftg{G})=\ft$. Also, if $\ftg{H}\in X_{\alpha\beta}$ is such that $\ftg{H}\neq \ftg{G}$, then as above, there exist $D\in \ftg{H}\setminus\ftg{G}$ and $C\in\ft$ such that $r(C,\beta)\cap D=\emptyset$. This implies that $r(C,\beta)\notin \ftg{H}$ and $f_{\alpha[\beta]}(\ftg{H})\neq\ft$. Hence, the uniqueness of $\ftg{G}$.
	
\end{proof}

\begin{lemma}\label{lemma:from.filter.to.word}
	Let $\alpha\in \awstar$. For $\ft\in X_{\alpha}$, there is a unique $y\in F_{\alpha}$ such that $Z_{y_{0,n}}\cap F_{\alpha}\in \ft$ for all $n\in\nn$.
\end{lemma}

\begin{proof}
	We use the fact that if a union of elements of $\acfra$ is in $\ft$ then one of the sets of the union is in $\ft$ (since ultrafilters in Boolean algebras are prime filters). In particular, if $A_1,\ldots, A_n\in \acfra$ are such that the disjoint union $A_1\sqcup \cdots\sqcup A_n$ belongs to $\ft$ then there is a unique $i\in\{1,\ldots,n\}$ such that $A_i\in \ft$.
	
	Notice that
	\[\ft\ni r(\alpha)=\bigsqcup_{b\in\alf}Z_b\cap r(\alpha)\]
	so that there is a unique $b_0\in \alf$ such that $\emptyset \neq Z_{b_0}\cap r(\alpha)\in \ft$. Also, there exists $x^{(0)}\in\ssf$ such that $b_0x^{(0)}\in r(\alpha)$. Now
	\[\ft\ni Z_{b_0}\cap r(\alpha) = \bigsqcup_{b\in\alf:b_0b\in L_{\ssf}}Z_{b_0b}\cap r(\alpha),\]
	and there is a unique $b_1\in \alf$ such that $b_0b_1\in L_{\ssf}$ and $Z_{b_0b_1}\cap r(\alpha)\in\ft$. As above, there exists $x^{(1)}\in \ssf$ such that $b_0b_1x^{(1)}\in r(\alpha)$. Continuing this process we find sequences $\{b_n\}_{n\in\nn}$ in $\alf$ and $\{x^{(n)}\}_{n\in\nn}$ in $\ssf$ such that for all $n\in\nn$, $b_0\ldots b_n \in L_{\ssf}$, $Z_{b_0\ldots b_n}\cap r(\alpha)\in \ft$ and $b_0\ldots b_nx^{(n)}\in r(\alpha)$. Define $y=b_0b_1\ldots$ and observe that $b_0\ldots b_nx^{(n)}\xrightarrow{n\to\infty} y$ so that $y\in \ssf$. Since $r(\alpha)=F_{\alpha}$ is closed, $y \in F_{\alpha}$.
	
	The uniqueness of $y$ follows by construction.
\end{proof}

\begin{lemma}\label{lemma:follower.set.dense}
	Let $\alpha\in \awstar$. For all $y\in F_{\alpha}$, the set $\ft_y:=\{A\in\acfra\mid y\in A\}$ is an ultrafilter in $\acfrg{\alpha}$. Moreover, the map $\mathcal{I}:y\in F_{\alpha}\to \ft_y\in X_{\alpha}$ is one-to-one and has dense image.
\end{lemma}

\begin{proof}
	Since $y\in F_{\alpha}=r(\alpha)$, for an arbitrary $A\in\acfra$, either $y\in A$ or $y\notin A$ so that $A\in \ft_y$ or $r(\alpha)\setminus A\in\ft_y$. This implies that $\ft_y$ is an ultrafilter in the Boolean algebra $\acfra$.
	
	Suppose now that $y,z\in F_{\alpha}$ and $y\neq z$. Then there exists $n\in\nn$ such that $y_n\neq z_n$. Define $\beta = y_0\ldots y_n$ and $\gamma = z_0\ldots z_n$. For $A=Z_{\beta}\cap F_{\alpha}\in \acfra$, we have that $y \in A$ but $z \notin A$, which implies that $\ft_y\neq \ft_z$.
	
	Finally, given any basic open set $U_A$ of $X_{\alpha}$, where $A \in\acfra$, we have that $A\scj F_{\alpha}$ and $\ft_y\in U_A$ for all $y\in A$. This implies the density of the image of $\mathcal{I}$.
\end{proof}

\begin{proposition} \label{prop:tightStoneDualHomeo}
	The map $\Phi:\ftight\to X_{\eword}$ given by $\Phi(\xi)=\xi_0$ is a homeomorphism.
\end{proposition}

\begin{proof}
	As observed at the start of this section, the elements of $\ftight$ are all of infinite type. Hence, for $\xi=\xia\in\ftight$, the labelled path $\alpha$ is an element of $\ssf$. According to \cite[Proposition 5.9]{MR3648984}, $\xi_0$ can be either the empty set or an element of $X_{\eword}$. The case of being the empty set only happens if $r(A,\alpha_1)$ is empty for all $A\in\acf$. However $r(Z_{\alpha_1},\alpha_1)=F_{\alpha_1}\neq \emptyset$, which implies that $\xi_0\in\ X_{\eword}$ and that $\Phi$ is well defined.

	To prove that $\Phi$ is bijective, we build its inverse. For $\ft\in X_{\eword}$, let $y\in\ssf$ be as in Lemma \ref{lemma:from.filter.to.word}. For each $n\in\nn$, let $\ft_n$ be the unique element of $X_{y_{0,n}}$ such that $f_{\eword[y_{0,n}]}(\ft_n)=\ft$ as in Lemma \ref{lemma:f.almost.onto}. Then $\{\ft_n\}_{n\in\nn}$ is a complete family of ultrafilters for $y$ (see Section \ref{subsection:filters.E(S)}). Using \cite[Propositions 4.9 and 5.9]{MR3648984}, we can associate a unique element $\xi_{\ft}\in\ftight$ to the pair $(y,\{\ft_n\}_{n\in\nn})$. Define $\Psi:X_{\eword} \to \ftight$ by $\Psi(\ft)=\xi_{\ft}$.
	
	By construction, $\Phi(\Psi(\ft))=\ft$ for all $\ft\in X_{\eword}$. That $\Psi(\Phi(\xia))=\xia$, for $\xia\in\ftight$, follows from the uniqueness in Lemma \ref{lemma:f.almost.onto} and that $f_{\eword[\alpha_{0,n}]}(\xi_n)=\xi_0$ for all $n\in\nn$.

	We now prove that $\Phi$ is a homeomorphism. Consider $e=(\alpha,A,\alpha)\in E(S)$ and $B\in\acf$ such that $r(B,\alpha)=A$, which exists by Lemma \ref{lemma:rel.range.shift.onto}. By looking at the constructed correspondence between $\ftight$ and $X_{\eword}$, we see that for $\xi\in\ftight$, we have that $e\in\xi$ if and only if $(\eword,B,\eword)\in\xi$, which implies that $V_{e}=V_{(\eword,B,\eword)}$. 
	
	Now, for $B\in\acf$, the set $U_B=\{\ft\in X_{\eword}\mid B\in\ft\}$ is a basic open subset of $X_{\eword}$. Then $\Phi$ sends the family of basic open sets $\{V_{(\eword,B,\eword)}\mid B\in\acf\}$ of $\ftight$ to the family of basic open sets $\{U_{B}\mid B\in\acf\}$, and $\Psi$ goes the other way around. Whence $\Phi$ is a homeomorphism.
\end{proof}

\begin{proposition}
	The algebra $\mathcal{D}_{\ssf}$ is isomorphic to $C_0(\ftight)$.
\end{proposition}

\begin{proof}
	As discussed above $\mathcal{D}_{\ssf}\cong\overline{\text{span}}\{p_A\mid A\in\acf_{\ssf}\}$. On the other hand, by \cite[Theorem 6.9]{MR3680957}, $C_0(\ftight)\cong\overline{\text{span}}\{s_{\alpha}p_As_{\alpha}^*\mid (\alpha,A,\alpha)\in S(\dgraph_{\ssf}, \lbf_{\ssf}, \acf_{\ssf})\}$. It then suffices to show that if $(\alpha,A,\alpha)\in S(\dgraph_{\ssf}, \lbf_{\ssf}, \acf_{\ssf})$, then $s_{\alpha}p_As_{\alpha}^*=p_B$ for some some $B\in\acf_{\ssf}$.
	
	By Theorem \ref{thm:IsomorphismThm} and Lemma \ref{lem:crossProdComputations} (\ref{lem:crossProdComputations_triple}), $s_{\alpha}p_As_{\alpha}^*=p_B$ if and only if $V_{(\alpha,A,\alpha)}=V_{(\eword,A,\eword)}$. Suppose then that $(\alpha,A,\alpha)\in S(\dgraph_{\ssf}, \lbf_{\ssf}, \acf_{\ssf})$ and notice that $A\in\acfra$. By Lemma \ref{lemma:rel.range.shift.onto}, there exists $B\in\acf_{\eword}$ such that $r(B,\alpha)=A$. Since $r(Z_{\alpha},\alpha)=F_{\alpha}$ and the labelled space is weakly left resolving, $r(B\cap Z_{\alpha}, \alpha)=A\cap F_{\alpha}=A$, so we may assume without loss of generality that $B\scj Z_{\alpha}$. Using the definition of complete family and the proof of Proposition \ref{prop:tightStoneDualHomeo}, we have that
	\[\xi\in V_{(\alpha,A,\alpha)}\Leftrightarrow (\alpha,A,\alpha)\in\xi_{|\alpha|}\Leftrightarrow (\eword,B,\eword)\in\xi_0\Leftrightarrow \xi\in V_{(\eword,B,\eword)}.\]
	
	The result follows.
\end{proof}

\begin{remark}
	We can use the theory developed here and in \cite{Gil3} to describe $C^*(\ssf)$ as a groupoid C*-algebra and as a partial crossed product. Using Proposition \ref{prop:tightStoneDualHomeo} and the map $\sigma$ of \cite[Proposition 4.8]{Gil3}, we can define a local homeomorphism $\tilde{\sigma}:X_{\eword}\to X_{\eword}$ as follows: for $\ft\in X_{\eword}$ and $y\in\ssf$ given by Lemma \ref{lemma:from.filter.to.word}, we define \[\tilde{\sigma}(\ft)=\Phi(H_{[y_0]y_1y_2\ldots}(\Phi^{-1}(\ft)))=\usetr{\{r(A,y_0)\mid A\in\ft\}}{\acf}.\]
	From the results of \cite{Gil3}, if $\mathcal{G}(X_\eword,\tilde{\sigma})$ is the groupoid defined from the pair $(X_\eword,\tilde{\sigma})$ as in \cite{MR1770333}, then $C^*(\ssf)\cong C^*(\mathcal{G}(X_\eword,\tilde{\sigma}))$.
	
	A similar description could be done for the partial action, however this will not be necessary for the remainder of the text.
\end{remark}

We now tackle the problem of studying the simplicity of $C^*(\ssf)$. We recover some of the results of \cite{MR3639522} using the theory developed in Section \ref{section:simplicity}. We begin by characterizing minimality of the partial action (see also Theorem \ref{thm:minimality.labelled.spaces}) in terms of hyper cofinality. See \cite[Section 13]{MR3639522} for the relevant discussion and motivation on cofinality. 


\begin{definition}
	For a finite set $P\scj L_{\ssf}$, $F_P:=\bigcap_{\beta\in P}F_{\beta}$ is called the \emph{follower set} of $P$. Given also $x\in\ssf$, the \emph{cost} of reaching $x$ from $P$ is
	\[\text{Cost}(P,x)=\inf\{|\alpha|+|\gamma|: x=\alpha y \in \ssf,\ \beta \gamma y \in \ssf,\text{ for all }\beta\in P\},\]
	with the convention that $\inf\emptyset=\infty$. We say that the subshift is \emph{hyper cofinal} if
	\[\sup_{x\in \ssf}\text{Cost}(P,x)<\infty,\]
	for every $P\scj L_{\ssf}$ finite such that $F_P\neq\emptyset$.
\end{definition}

\begin{proposition}\label{prop:minimality.shift}
	The only hereditary saturated subsets of $\acf_{\ssf}$ are $\{\emptyset\}$ and $\acf_{\ssf}$ if and only if $\ssf$ is hyper cofinal.
\end{proposition}

\begin{proof}
	Suppose that the only saturated hereditary subsets of $\acf_{\ssf}$ are $\{\emptyset\}$ and $\acf_{\ssf}$, and let $P\scj L_{\ssf}$ be a finite set such that $F_P\neq\emptyset$. Since $\sat(\her(F_P))$ is a hereditary saturated subset of $\acf_{\ssf}$ that contains $F_P$, we have that $\sat(\her(F_P))=\acf_{\ssf}$.
	
	Since $\ssf\in\acf_{\ssf}$ is regular, there exists $n\in\nn$ such that $\ssf\in \sat^{[n]}(\her(F_P))$, which means that for every $\alpha\in L_{\ssf}$ such that $|\alpha|=n$, we have that $r(X,\alpha)=F_{\alpha}\in \her(F_P)$. Therefore, there exist $\gamma_{\alpha}^1,\ldots,\gamma_{\alpha}^{m_\alpha}$ such that $F_{\alpha}\scj \bigcup_{i=1}^{m_{\alpha}}r(F_P,\gamma^i)$. In particular, if $x=\alpha y\in\ssf$, then $y\in \bigcup_{i=1}^{m_{\alpha}}r(F_P,\gamma^i)$ and hence $y\in r(F_P,\gamma^i)$ for some $i\in\{1,\ldots,m_{\alpha}\}$, which implies that $\beta\gamma^i y\in\ssf$ for all $\beta\in P$. It follows that $\text{Cost}(P,x)\leq n+\max\{|\gamma_{\alpha}^1|,\ldots,|\gamma_{\alpha}^{m_\alpha}|\}$. Since there are only finitely many labelled paths of length $n$,
	\[\sup_{x\in \ssf}\text{Cost}(P,x)\leq n+\max\{|\gamma^i_\alpha| : \alpha\in\ssf, |\alpha|=n,i=1,\ldots,m_{\alpha}\}<\infty.\]
	
	Now suppose that $\ssf$ is hyper cofinal. We need to prove that for all $A\in\acf_{\ssf}$, if $A\neq\emptyset$, then $\sat(\her(A))=\acf_{\ssf}$. We will make a series of reductions to show that it is sufficient to consider $A=F_P$ for some $P\scj L_{\ssf}$ finite. For that, observe that for $A,B\in\acf_{\ssf}$, if $A\scj B$ then $\sat(\her(A))\scj\sat(\her(B))$, and for any  $\alpha\in L_{\ssf}$, $\sat(\her(r(A,\alpha)))\scj \sat(\her(A))$. Since $\acf_{\ssf}$ is generated by $C(\alpha,\beta)$, it follows that any element of $\acf_{\ssf}$ is a finite union of sets of the form
	\[C(\alpha_1,\beta_1)\cap\cdots\cap C(\alpha_n,\beta_n)\cap C(\mu_1,\nu_1)^c\cap \cdots C(\mu_m,\nu_m)^c.\]
	If $z$ is an element in this intersection, then for $p>\max\{|\beta_1|,\ldots,|\beta_n|,|\nu_1|,\ldots,|\nu_m|\}$,
	the relative range of this intersection with respect to $z_{0,p}$ is of the form
	\[F_{\gamma_1}\cap\cdots\cap F_{\gamma_k}\cap F_{\delta_1}^c\cap\cdots\cap F_{\delta_l}^c,\]
	and it contains $z_{p+1,\infty}$.
	For $y$ in this intersection, there exists then $q\in\nn^*$ such that $\delta_1y_{0,q},\ldots\delta_ly_{0,q}\notin L_{\ssf}$, which implies that
	\[F_{\gamma_1}\cap\cdots\cap F_{\gamma_k}\cap Z_{y_{0,q}}\scj F_{\gamma_1}\cap\cdots\cap F_{\gamma_k}\cap F_{\delta_1}^c\cap\cdots\cap F_{\delta_l}^c.\]
	Finally, $y_{q+1,\infty}\in r(F_{\gamma_1}\cap\cdots\cap F_{\gamma_k}\cap Z_{y_{0,q}},y_{0,q})=F_P$ for some finite $P\scj L_{\ssf}$.
	
	Hence suppose that $P\scj L_{\ssf}$ is finite and that $F_P\neq \emptyset$. By hypothesis, we have that $N:=\sup_{x\in \ssf}\text{Cost}(P,x)<\infty$. Then, given $x\in\ssf$, there exists $\alpha,\gamma\in L_{\ssf}$ with $|\alpha|,|\gamma|\leq N$ such that $x=\alpha y$ and $\beta\gamma y\in\ssf$ for all $\beta\in P$. We can write $y=\alpha'y'$ in such a way that $|\alpha|+|\alpha'|=N$, and then $\beta\gamma\alpha' y'\in\ssf$ for all $\beta\in P$. This implies that if $|\alpha|=N$, then $F_{\alpha}\scj \bigcup_{\gamma\in L_{\ssf},|\gamma|\leq 2N}r(F_P,\gamma)$, which is a finite union. It follows that $r(\ssf,\alpha)=F_{\alpha}\in\her(F_P)$ and since alpha was arbitrary, $\ssf\in \sat^{[N]}(\her(F_P))$. Finally, $\ssf$ is the top element of $\acf_{\ssf}$ and by the definition of a hereditary subset, $\sat(\her(F_P))=\acf_{\ssf}$.
\end{proof}

\begin{definition}
	We say that $\gamma\in L_{\ssf}$ is a \emph{circuit} if $\gamma^{\infty}\in\ssf$. We say that the subshift $\ssf$ satisfies \emph{condition (L)} if for every finite $P\scj L_{\ssf}$ such that $\gamma^\infty\in F_P$ for some circuit $\gamma$, there is an element $y\in F_P$ that is different from $\gamma^\infty$.
\end{definition}

Our definition of condition (L) is the same as \cite[Theorem 12.6(ii)]{MR3639522}. We call it condition (L) due to the following proposition.

\begin{proposition}\label{prop:top.principal.shift}
	The subshift $\ssf$ satisfies condition (L) if and only if $(\dgraph_{\ssf}, \lbf_{\ssf}, \acf_{\ssf})$ satisfies condition ($L_\acf$).
\end{proposition}

\begin{proof}
	If $\ssf$ does not satisfy condition (L), then there exists $P\scj L_{\ssf}$ finite and $\gamma$ a circuit such that $F_P=\{\gamma^{\infty}\}$. In this case, $(\gamma,F_P)$ is a cycle without exit, so that $(\dgraph_{\ssf}, \lbf_{\ssf}, \acf_{\ssf})$ does not satisfy condition ($L_\acf$).
	
	Conversely, if $(\dgraph_{\ssf}, \lbf_{\ssf}, \acf_{\ssf})$ does not satisfy condition ($L_\acf$), then there exists a cycle without exit $(\alpha,A)$. However, for the shift labelled space, this implies that $\alpha$ is a circuit and $A=\{\alpha^{\infty}\}$. Arguing as in the proof of Proposition \ref{prop:minimality.shift}, for $n\in\nn$ large enough, there exists $P\scj L_{\ssf}$ finite such that $\alpha^{\infty}\in F_P\scj r(A,\alpha^n)=A=\{\alpha^{\infty}\}$. If follows that $\ssf$ does not satisfy condition (L).
\end{proof}

The following theorem is the equivalence of (i) and (ii) from \cite[Theorem 14.5]{MR3639522}, which we prove using the theory developed in this paper.

\begin{theorem}
	Let $\ssf$ be a subshift. Then, $C^*(\ssf)$ is simple if and only if $\ssf$ is hyper cofinal and satisfies condition (L).
\end{theorem}

\begin{proof}
	If follows from Theorem \ref{theorem:simple.labelled.space.algebra} and Propositions \ref{prop:minimality.shift} and \ref{prop:top.principal.shift}.	
\end{proof}

We now  compare our results with   \cite[Example 11.4]{MR3606190}. First we need some definitions. Given $l\in\nn$, we say that $x,y\in\ssf$ are \emph{$l$-past equivalent}, written $x\sim_l y$, if
\[\{\alpha\in L_{\ssf}\mid x\in F_{\alpha},|\alpha|\leq l\}=\{\alpha\in L_{\ssf}\mid y\in F_{\alpha},|\alpha|\leq l\}.\]
The shift $\ssf$ is said to be \emph{cofinal in past equivalence} if for any $x,y\in\ssf$ and $l\in\nn$, there exist $z\in\ssf$, $m,n\in\nn$ such that $y\sim_l z$ and $\sigma^m(x)=\sigma^n(z)$. A point $x\in\ssf$ is said to be \emph{isolated in past equivalence} if there exists $l\in\nn$ such that $[x]_l=\{x\}$, and  \textit{cyclic} if $x=\gamma^{\infty}$ for some circuit $\gamma$. It is claimed in \cite[Example 11.4]{MR3606190} that $C^*(\ssf)$ is simple if and only if $\ssf$ is cofinal in past equivalence and there is no cyclic point isolated in past equivalence. However, we illustrate with a counterexample that a stronger condition than cofinality is needed.

First, observe that there is no cyclic point isolated in past equivalence if and only if $\ssf$ satisfies condition (L). Indeed, for a circuit $\gamma$, we just need to look at the set $P=\{\alpha\in L_{\ssf}\mid \gamma^{\infty}\in F_{\alpha},|\alpha|\leq l\}$ for $l\in\nn$. We now check that being cofinal in past equivalence is not equivalent to being hyper cofinal. For that, consider the shift over the alphabet $\{a,b,c\}$ with set of forbidden words:
\[\mathcal{F}=\{ba^kb\mid k\text{ is not a power of }2\}\cup\{ca^kb\mid k\text{ is not a power of }2\}.\]
The class of a point in $y\in\ssf_{\mathcal{F}}$ depends on whether the beginning of $y$ is of the form $a^kb$, for some $k$, or not, and one can check that each equivalence class with respect to $\sim_l$ is infinite,  so that there are no cyclic points isolated in past equivalence. Notice that given $x\in \ssf_{\mathcal{F}}$, there exists $m\in\nn$ such that we can glue $a^kb$ or $c$ at the beginning of $\sigma^m(x)$. This implies that $\ssf_{\mathcal{F}}$ is cofinal in past equivalence.  However,  we claim that $\ssf_{\mathcal{F}}$ is not hyper cofinal. To see this claim,  consider $P=\{c\}$ and $x=a^kb^{\infty}$, where $k=2^n+2^{n+1}$ for some $n\in\nn$. Then $\text{Cost}(P,x)=2^n$, which implies that $\ssf_{\mathcal{F}}$ is not hyper cofinal.

The issue with the proof in \cite[Example 11.4]{MR3606190} is the claim that to check minimality, using our terminology, it is sufficient to show that for each $x\in\ssf=F_{\eword}$, the orbit of each $\ft_x$, given in Lemma \ref{lemma:follower.set.dense}, is dense. However, we also need a boundedness condition on $m$ and $n$ in the definition of cofinal in past equivalence for it to be equivalent to minimality. Although we did not use the density of $\{\ft_x\mid x\in\ssf\}$ in $X_{\eword}$ in our proof of minimality, the proof of \cite[Theorem 13.6]{MR3639522} uses an analogous property where the necessity of a boundedness condition is evident.

\bibliographystyle{abbrv}
\bibliography{labelledpartcrossedprod_ref}

\end{document}